\documentclass[11pt,twoside]{article}
\usepackage{relsize,amsmath,amsfonts,amssymb,amsthm,graphicx,epsf,epstopdf,
color,pifont,subfigure,flafter,bm,amscd,multirow,setspace,tikz,caption,mathrsfs}
\allowdisplaybreaks[4]

\newtheorem{df}{Definition}[section]
\newtheorem{lm}{Lemma}[section]
\newtheorem{thm}{Theorem}[section]

\newtheorem{prop}{{\bf Proposition}}[section]

\newcounter{saveeqn}%

\makeatletter
\evensidemargin
\oddsidemargin
\makeatother

\title{\Large\bf Classifications and bifurcations of tangent points and their loops
of planar piecewise-smooth systems
\thanks{
Supported by NSFC \#12271378.
}
}

\author{Zhihao Fang, ~~Xingwu Chen\!\!
\footnote{Corresponding author. Email address: scuxchen@scu.edu.cn, xingwu.chen@hotmail.com}
\\
{\small School of Mathematics, Sichuan University,}
{\small Chengdu, Sichuan 610064, P. R. China}
}
\date{}

\begin{document}
\maketitle

%%%%%%%%%%%%%%%%%%%%%%%%%%%%%%%%%%%%%%%%%%%%%%%%%%%%%%%%%%%%%%%%%%%%%%%

\begin{abstract}
Tangent points, especial dynamics existing only in piecewise-smooth systems,
usually have dynamical properties like equilibria of smooth systems.
Loops connecting tangent points own partly properties of limit cycles and homoclinic loops
of smooth systems. In this paper we give classifications for tangent points by tangency degree and
for loops connecting them by configuration,
and investigate their bifurcations.
The classic method is to construct functional parameters
for the case of low tangency degree but, is no longer valid for the case of general tangency degree,
which leads to complicated interlacement of sliding and crossing motions on the switching manifold.
We provide an explicit unfolding for tangent points of general tangency degree and their
loops, in which explicit functional functions are constructed to replace
functional parameters.
We mainly obtain relations between original tangency degree and numbers of bifurcating
tangent points, bifurcating tangent orbits and bifurcating loops for this unfolding.
Some of these relations are generalizations to general tangency degree and others are new for previous publications.

\vskip 0.2cm


\end{abstract}

\baselineskip 15pt
\parskip 10pt
\thispagestyle{empty}
\setcounter{page}{1}

\newpage
\tableofcontents

%%%%%%%%%%%%%%%%%%%%%%%%%%%%%%%%%%%%%%%%%%%%%%%%%%%%%%%%%%%%
\section{Introduction}

\setcounter{equation}{0}
\setcounter{lm}{0}
\setcounter{thm}{0}
\setcounter{rmk}{0}
\setcounter{df}{0}
\setcounter{cor}{0}
\setcounter{prop}{0}

Piecewise-smooth systems (abbreviated as PWS systems usually) widely appear in many fields as indicated in
\cite{Bernardo08,Bernardo082}. For instance, PWS systems are used to model mechanics systems with dry friction
(see, e.g., \cite{Galvanetto01,Galvanetto99}), control systems as DC-DC converters (see, e.g., \cite{CP,CTV,Schild09})
and biological systems as predator-prey models (see, e.g., \cite{TangS13,TangS132}). Thus,
PWS systems attract much attentions of researchers (see, e.g., textbooks \cite{Bernardo08,Filippov88,HuangGW,Kunze}) and need a great development.

Consider a PWS system defined on a bounded open set ${\cal U}\subset\mathbb{R}^2$ containing the origin $O:(0,0)$
\begin{equation}
\begin{aligned}
		\left( \begin{array}{c}
		\dot{x}\\
		\dot{y}\\
	\end{array} \right) =\left\{
    \begin{aligned}
&		\left( \begin{array}{c}
			f^+(x,y)\\
			g^+(x,y)
		\end{array} \right)   &&\mathrm{if}~(x,y)\in\Sigma^+,\\
&		\left( \begin{array}{c}
			f^-(x,y)\\
			g^-(x,y)
		\end{array} \right)   && \mathrm{if}~(x,y)\in\Sigma^-,
	\end{aligned} \right.
    \label{pws1}
\end{aligned}
\end{equation}
where $\dot x:=dx/dt, \dot y:=dy/dt$, $f^\pm,g^\pm$ are $C^{\infty}$ functions on $\mathbb{R}^2$,
$\Sigma^+:=\left\{(x,y)\in{\cal U}:~y>0\right\}$ and $\Sigma^-:=\left\{(x,y)\in{\cal U}:~y<0\right\}$.
System~\eqref{pws1} consists of the {\it upper subsystem} defined on $\Sigma^+$,
the {\it lower subsystem} defined on $\Sigma^-$ and a {\it switching manifold}
$\Sigma:=\left\{(x,y)\in{\cal U}:~y=0\right\}$.
An equilibrium of a subsystem is called a {\it standard equilibrium} (resp. {\it boundary equilibrium})
of system~\eqref{pws1} if it lies outside of $\Sigma$ (resp. on $\Sigma$).

For $p:(x,0)\in \Sigma$, if $h(x):=g^+(x,0)g^-(x,0)>0$, the orbit of one subsystem reaching $p$
connects the orbit of the other subsystem escaping from $p$.
Thus, there is an orbit of system~\eqref{pws1} crossing $\Sigma$ at $p$
and set
$\Sigma_c:=\left\{(x,0)\in\Sigma:~h(x)>0\right\}$
is called a {\it crossing region}. If $h(x)<0$,
as in \cite{Filippov88} a {\it sliding vector field}
on {\it sliding region}
$\Sigma_s:=\left\{(x,0)\in\Sigma:~h(x)<0\right\}$
is defined as
\begin{equation}
X_s(x,0):= \left(\frac{f^+(x,0)g^-(x,0)-f^-(x,0)g^+(x,0)}{g^-(x,0)-g^+(x,0)},0 \right)^\top.
\label{SVF}
\end{equation}
Thus, orbit of subsystem reaching $p\in\Sigma_s$ connects a sliding orbit.
Point $(x_0,0)\in \Sigma_s$ is called a {\it pseudo equilibrium} of system~\eqref{pws1} if $X_s(x_0,0)={\boldsymbol 0}$.

Point $(x_0,0)\in \Sigma$ is called a {\it tangent point} of system~\eqref{pws1} if
$h(x_0)=0$ and $f^{\pm}(x_0,0)\ne0$,
i.e., there is at least one orbit of subsystems which is tangent with $\Sigma$ at $(x_0,0)$.
Different from smooth systems, as an especial point in PWS systems tangent point influences
the dynamics greatly near itself.
For instance, when a tangent point corresponds to a zero of $h(x)$ of high multiplicity, it may break into several
tangent points under perturbations (see, e.g., \cite{Fang21,Ponce15,Teixeira11,Kuznetsov03,Han13}).
Tangent points may bifurcate some crossing limit cycles, like foci of classic smooth systems (see, e.g., \cite{Coll01,Ponce22,Filippov88,Zhang10,Kuznetsov03,Han12}).
This phenomenon is called a pseudo-Hopf
bifurcation of PWS systems in \cite{Kuznetsov03} or directly a Hopf bifurcation in \cite{Zhang10}.
On the other hand, the orbits of subsystems of PWS systems may form a loop with tangent points together.
Such loop is usually structurally unstable and crossing limit cycles may bifurcate from it under perturbations
(see, e.g., \cite{Ponce15,Kuznetsov03,Han13}). Such loop is called a critical crossing limit cycle as in \cite{Ponce15,Kuznetsov03} or a homoclinic loop as in \cite{Han13}. The former means that such loop owns partly properties of crossing limit cycles and the latter
 means that tangent points are usually regarded as a kind of singular points of PWS systems.
  Thus, in this paper we generally call it a {\it nonsliding loop connecting tangent points} as defined in section 2.
  By the bifurcation theory, tangent points and their nonsliding loops are very helpful in the
  investigation of crossing limit cycles in PWS systems. For example, it is proved for piecewise-linear systems
   that a lower bound of the maximum number of crossing limit cycles is $2$ (see \cite{Zhang10}) and later $3$
  (see, e.g., \cite{Braga12,Buzzi13,Ponce14,Yang12}).
  The bifurcation of  tangent points and their nonsliding loops
  is also analyzed for PWS Li\'enard systems (see, e.g., \cite{ChenH18,Guan22})
  and PWS predator-prey models (see, e.g., \cite{TangS13,TangS132}).

Results of tangent points and their nonsliding loops in all previous publications are under the requirement that
the tangency degree between orbits passing through tangent points and $\Sigma$ is low, i.e.,
tangent points correspond to zeros of $h(x)$ no more than multiplicity 2. A natural question is as follows.
\\
(Q)~{\it How about the bifurcations of tangent points and their nonsliding loops when the tangency degree is high or general?}

In this paper, we focus on this question.
For the case that tangent points are of low tangency degree, the classic method
is to construct functional parameters to analyze bifurcations,
where functional parameters are the parameters determining some dynamical behaviors.
For example, a functional parameter $\alpha$ is constructed in \cite{Li20}
such that a limit cycle of one subsystem is tangent with $\Sigma$ for $\alpha=0$, away
from $\Sigma$ for $\alpha<0$ and intersects $\Sigma$ at exactly two points for $\alpha>0$.
Similar idea can be found in \cite{Bonet18,Fang21,Teixeira11,Kuznetsov03,Novaes18,Huang22}.
Unfortunately, such idea is not enough to unfold plentiful dynamics sufficiently for question (Q)
because of the generality of tangency degree.
We first give a series definitions and complete classifications
for general tangent points and loops connecting them,
and construct their explicit unfolding systems, in which functional functions are introduced to replace
classic functional parameters.
Analyzing unfolding systems, we answer question (Q) including the numbers of tangent points,
orbits passing through tangent points and loops under perturbation. Relations between
these numbers and tangency degree of original tangent points are given.
Here ``loops'' concludes crossing limit cycles, sliding and nonsliding loops connecting tangent points as
defined in section 2.

This paper is organized as follows. Definitions and classifications for tangent
points and their loops are stated in section 2 as well as literature reviews.
Main results are stated in section 3 and functional functions in unfolding systems
are analyzed in section 4. Proofs of main results
on bifurcations of tangent points
and on bifurcations of their loops are stated in section 5 and section 6 respectively.
Finally, we summarize conclusions and give some discussion remarks in section 7 to end this paper.

%%%%%%%%%%%%%%%%%%%%%%%%%%%
\section{Classifications of tangent points and their loops}

\setcounter{equation}{0}
\setcounter{lm}{0}
\setcounter{thm}{0}
\setcounter{rmk}{0}
\setcounter{df}{0}
\setcounter{cor}{0}

In this section we give some preliminaries and classify tangent points and loops connecting them, the
research object of this paper.
The solution of system~\eqref{pws1} is defined in \cite{Filippov88} by {\it Filippov Conventions}.
A continuous function $(x(t),y(t))^\top$ defined over some interval $I$ is called a {\it solution} of system~\eqref{pws1}
if it satisfies differential inclusion
$
(\dot x,\dot y)^\top\in F(x,y)
$
almost everywhere over $I$, where
\begin{equation*}
F(x,y):=\left\{
\begin{aligned}
&\left\{{\cal Z}^+(x,y)\right\},  &&(x,y)\in\Sigma^+, \\
&\left\{a{\cal Z}^+(x,y)+(1-a){\cal Z}^-(x,y):~a\in[0,1]\right\}, &&(x,y)\in\Sigma, \\
&\left\{{\cal Z}^-(x,y)\right\},  &&(x,y)\in\Sigma^-
\end{aligned}
\right.
\end{equation*}
and ${\cal Z}^\pm(x,y):=\left(f^\pm(x,y),g^\pm(x,y)\right)$ denotes the upper and lower vector field of system~\eqref{pws1} respectively.
It is not difficult to find that differential inclusion defined on $\Sigma^+$ (resp. $\Sigma^-$) is equivalent to the upper (resp. lower) subsystem.

For tangent points, we come up with the following definition.

\begin{df} Tangent point $p:(x_0,0)$ of \eqref{pws1} is called of multiplicity $(m^+,m^-)$, if
\begin{eqnarray*}
&&g^+(x_0,0)=\frac{\partial g^+}{\partial x}(x_0,0)=...=\frac{\partial^{(m^+-1)}g^+}{\partial x^{(m^+-1)}}(x_0,0)=0,~\frac{\partial^{m^+}g^+}{\partial x^{m^+}}(x_0,0)\ne 0,\\
&&g^-(x_0,0)=\frac{\partial g^-}{\partial x}(x_0,0)=...=\frac{\partial^{(m^--1)}g^-}{\partial x^{(m^--1)}}(x_0,0)=0,~\frac{\partial^{m^-}g^-}{\partial x^{m^-}}(x_0,0)\ne 0,
\end{eqnarray*}
where non-negative integers $m^\pm$ satisfy $m^++m^-\ge 1$. Moreover, $m^{\pm}$ are called the
tangency multiplicities of the upper and lower subsystems respectively.
\label{df-tp}
\end{df}

$m^+=0,1,2$ (resp. $m^-=0,1,2$) in Definition~\ref{df-tp} correspond to that $p$ is a regular point, fold, cusp (see \cite{Teixeira11}) of the upper (resp. lower) subsystem respectively. Definition~\ref{df-tp} gives a classification for tangent points via the degeneracy, i.e., the multiplicity.
Clearly, multiplicity characterizes tangency degree between the orbit passing through $p$ (such orbit is called tangent orbit) and $\Sigma$. Consider a tangent point $p$ of the upper subsystem, we find that tangent orbit is separated into ``left'' and ``right'' segments by $p$. Further, $p$ is called a {\it visible} (resp. an {\it invisible}, a {\it left}, a {\it right}) tangent point of the upper subsystem if both (resp. neither, the left one, the right one) of two segments lie in $\Sigma^-$ and denote by V (resp. I, L, R) as shown in Figure~\ref{Fig-TP}.
Tangent point $p$ is of odd (resp. even) multiplicity if and only if $p$ is either visible or invisible (resp. either left or right). Thus, for tangent point $p$ of multiplicity $(m^+,m^-)$ with positive $m^\pm$ we call $p$ is a $VI$ tangent point if $p$ is visible for the upper subsystem and invisible for the lower subsystem. Similarly, we define $VV, VL, VR, IV, II, IL, IR, LV, LI, LL, LR, RV, RI, RL, RR$ tangent points and omit statements.

\begin{figure}[h]
\centering
\subfigure[visible]
 {
  \scalebox{0.34}[0.34]{
   \includegraphics{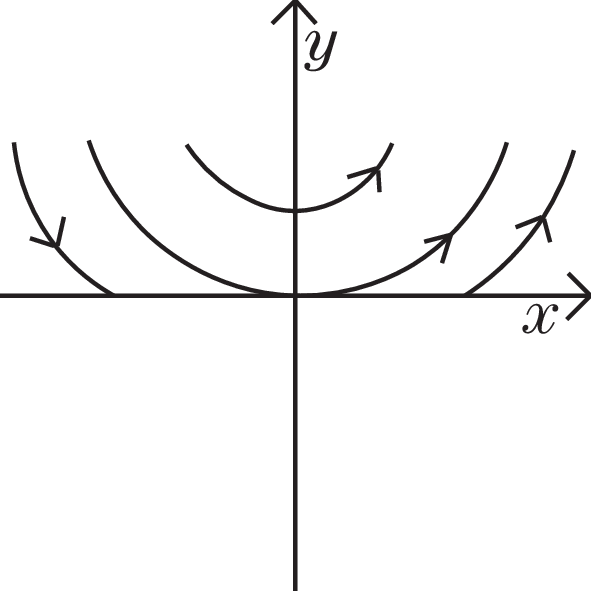}}}
\subfigure[invisible]
 {
  \scalebox{0.34}[0.34]{
   \includegraphics{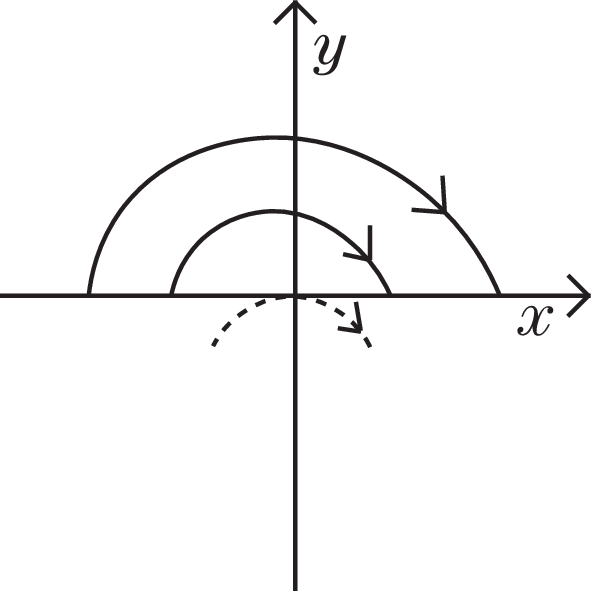}}}
\subfigure[left]
 {
  \scalebox{0.34}[0.34]{
   \includegraphics{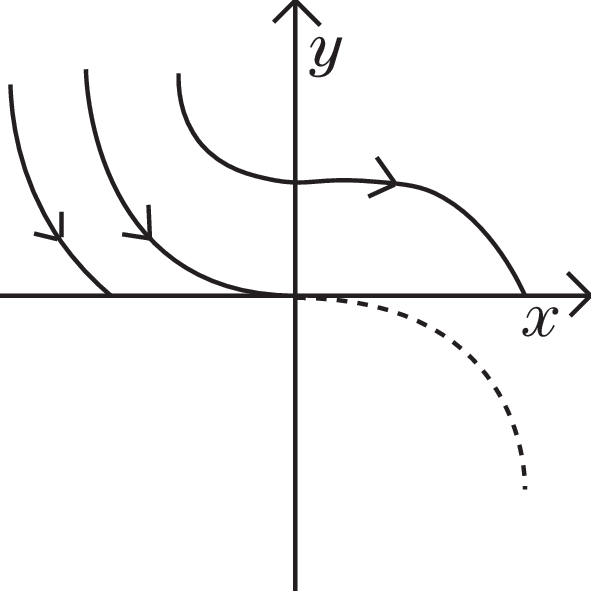}}}
\subfigure[right]
 {
  \scalebox{0.34}[0.34]{
   \includegraphics{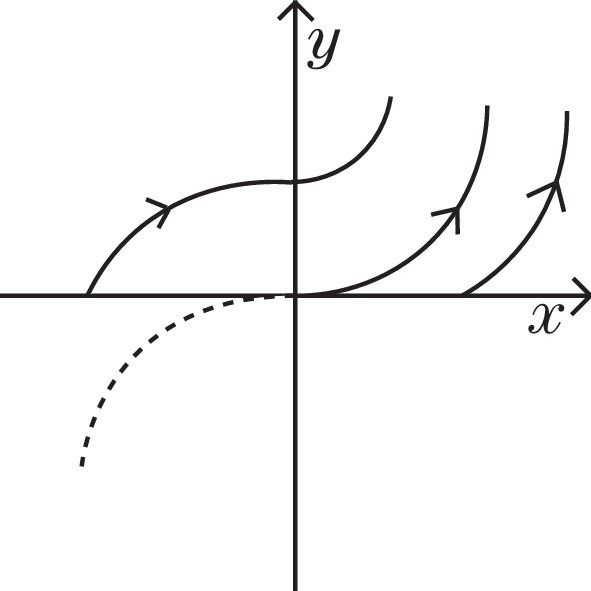}}}
   \caption{visibility of tangent point}
\label{Fig-TP}
\end{figure}

The dynamical behaviors of tangent points are rich and usually relate to pseudo equilibria, which is shown in the follows. Thus, unlike classic smooth systems, singular points of PWS systems conclude not only standard equilibria, pseudo equilibria and boundary equilibria.

As indicated in \cite{Ponce15,Han13}, tangent points of multiplicity $(0,1)$ are structurally stable, which means that there is no bifurcations happening under perturbations. Thus, tangent point $p$ is called to be {\it non-degenerate} when $m^++m^-=1$. As a contrary, tangent point $p$ is called to be {\it degenerate} when $m^++m^->1$. For tangent points of multiplicity $(0,2)$, it is proved in \cite{Kuznetsov03}
that at most two tangent points of multiplicity $(0,1)$ appearing under suitable perturbations.

In \cite{Kuznetsov03}, for tangent points of multiplicity $(1,1)$
an one-parametric unfolding is introduced to show the local bifurcations and
such tangent points are cataloged as $7$ kinds named $VV_1,~VV_2,~VI_1,~VI_2,~VI_3,~II_1,~II_2$
by the bifurcation phenomena. Moreover, as indicated in \cite{Bonet18} the local bifurcations given in \cite{Kuznetsov03} for $VV_1, VV_2, VI_1, II_1$ are complete.
For $II_2$ tangent point, there is at most one crossing limit cycle bifurcated from it under perturbations with some non-degenerate condition, which is proved in \cite{Bonet18,Filippov88, Kuznetsov03,Han12}. For the degenerate case of $II_2$, it is proved in \cite{Ponce22,Han12} that for any integer $k\ge 2$, there are perturbations with
exactly $k$ crossing limit cycles bifurcated from it.
Here degeneracy is with respect to bifurcations, not tangent points.
For $VI_2, VI_3$ tangent points, it is proved in \cite{Bonet18,Teixeira11,Kuznetsov03} that
one pseudo equilibrium is bifurcated out under some non-degeneracy and, then,
two pseudo equilibria are found for a two-parametric unfolding without the non-degeneracy in \cite{Fang21}.
By controlling the locations of bifurcating pseudo equilibrium and tangent point of the lower subsystem,
another bifurcation different from those $7$ kinds given in \cite{Kuznetsov03} is obtained for a tangent point of multiplicity $(1,1)$
in \cite{Siller} and such tangent point is named as $VI_4$.

For tangent points of multiplicity $(1,2)$, a two-parametric unfolding is analyzed in \cite{Teixeira11} for
local dynamics with small parameters and in \cite{Fang21} for global dynamics with general parameters, and
the uniqueness of bifurcating crossing limit cycles is proved as well as the existence of sliding loops connecting one tangent point and one pseudo equilibrium.

Besides tangent points, loops connecting tangent points are usually
structurally unstable as indicated in \cite{Bernardo08,Filippov88} so that we focus on not only local bifurcations for tangent points, but also nonlocal bifurcations for their loops.
Let $\Gamma_i(\alpha)$ denote a segment of a regular orbit of subsystem and be parameterized by $\alpha$, where $\alpha\in[\alpha_{i1},\alpha_{i2}],~\alpha_{i1}<\alpha_{i2}$, $i=1,2,...,n$, and satisfy that each $\Gamma_i(\alpha)$ only intersects with $\Sigma$ at endpoints $p_i$ and $p_{i+1}$, i.e.,
\begin{equation*}
\Gamma_i(\alpha_{i1})=p_{i}\in\Sigma,~~\Gamma_i(\alpha_{i2})=p_{i+1}\in\Sigma,~~\Gamma_i(\alpha)\notin\Sigma~{\rm for}~\alpha\in(\alpha_{i1},\alpha_{i2}).
\end{equation*}
Clearly, it is not difficult to find that $\Gamma_i(\alpha)$ is neither a standard equilibrium
nor a boundary equilibrium nor a sliding orbit. Thus, on any oriented Jordan curve consisting of $\Gamma_i(\alpha)$ ($i=1,...,n$)
all singular points are tangent points if they exist.

An oriented Jordan curve consisting of $\Gamma_i(\alpha)$ ($i=1,...,n$) is called a {\it crossing periodic orbit} if all $p_i$ ($i=1,...,n$) are crossing points. Specially, a crossing periodic orbit is called a {\it crossing limit cycle} if it is isolated, and denoted by $L_c$.

\begin{df}
An oriented Jordan curve consisting of $\Gamma_i(\alpha)$ ($i=1,...,n$) is called a nonsliding loop connecting tangent points of system~\eqref{pws1} if there exists $i^*\in \{1,...,n\}$ such that $p_{i^*}$ is a tangent point, and denoted by $L_{ns}$.
\label{df-nonsli}
\end{df}

\begin{df}
An oriented Jordan curve consisting of $\Gamma_i(\alpha)$ ($i=1,...,n$) and at least one regular sliding orbit is called a sliding loop connecting tangent points of system~\eqref{pws1}, and denoted by $L_s$.
\label{df-sli}
\end{df}

Clearly, there exist tangent points on $L_s$ because if not, regular sliding orbit goes to a pseudo equilibrium or infinity, which means that $L_s$ is not oriented or closed. We call $p_i$ ($i=1,...,n+1$) a {\it switching point} of $L_{ns}$ (resp. $L_s$) if $L_{ns}$ (res. $L_s$) enters one half-plane from the other half-plane at $p_i$ and denote the set of all switching points of $L_{ns}$ (res. $L_s$) by ${\mathcal P}_{ns}$ (res. ${\mathcal P}_{s}$). Any switching point on $L_{ns}$ or $L_s$ is either a crossing point or a tangent point.

\begin{df}
Loop $L_{ns}\in \Lambda$ is called to be {\it grazing} (resp. crossing, critical) if it has no switching points (resp. it has switching points and all of them are crossing points, it has switching points and some of them are tangent points), where $\Lambda$ is the set of all nonsliding loops connecting tangent points of system~\eqref{pws1}.
\label{df-loop}
\end{df}

By Definition~\ref{df-loop}, $\Lambda=\mathcal L^{gra}\cup \mathcal L^{cro} \cup \mathcal L^{cri}$, where
\begin{equation}
\begin{aligned}
&\mathcal L^{gra}:=\{L_{ns}\in \Lambda:~{\mathcal P}_{ns}=\emptyset\},~\\
&\mathcal L^{cro}:=\{L_{ns}\in \Lambda:~{\mathcal P}_{ns}\ne \emptyset, {\mathcal P}_{ns}\subseteq \Sigma_c \},\\
&\mathcal L^{cri}:=\{L_{ns}\in \Lambda:~{\mathcal P}_{ns}\ne \emptyset, {\mathcal P}_{ns}\setminus\Sigma_c\ne \emptyset\}.
\end{aligned}
\label{class-Tan}
\end{equation}

For any loop $L^{gra}\in\mathcal L^{gra}$, it is a standard periodic orbit of one subsystem which is tangent with $\Sigma$ one or more times. For $L^{gra}$ connecting one tangent point of multiplicity $(0,1)$, its bifurcations are analyzed in \cite{Kuznetsov03} and the existences of $L_s$ and standard limit cycles are proved as well as their coexistence. It is proved in \cite{Han13} that for $L^{gra}$ connecting one tangent point of multiplicity $(1,1)$ there exist one $L_c$ and one standard limit cycle under some perturbation. Later, the complete bifurcation diagram is given in \cite{Li20} via two functional parameters, and the coexistence of $2$ bifurcating $L_c$ and one standard limit cycle is proved.

The research method of bifurcations for loop $L^{cro}\in\mathcal L^{cro}$ is similar to $L^{gra}$ but, we still define it as a new type in (\ref{class-Tan}) instead of combining $L^{cro}$ and $L^{gra}$ together. The main reasons are that $L^{cro}$ is a crossing periodic orbit which is tangent with $\Sigma$ one or more times, and that there are $L^{cro}$ and no $L^{gra}$ bifurcated from $L^{cri}\in\mathcal L^{cri}$, which will be shown in later theorems.

For loop $L^{cri}\in\mathcal L^{cri}$, in \cite{Ponce14,Ponce15,Kuznetsov03, Han13} it is proved that
either one $L_c$ or one $L_s$ can be bifurcated from $L^{cri}$ connecting one tangent point of multiplicity $(0,1)$. For $L^{cri}$ connecting one tangent point of multiplicity $(1,1)$, at most two $L_c$ are bifurcated out in \cite{Han13} and, later, the complete bifurcation diagram is obtained in \cite{Novaes18,Huang22} via two functional parameters. For $L^{cri}$ connecting two tangent points of multiplicities $(1,0)$ and $(0,1)$, the complete diagram of two functional parameters is obtained in \cite{Huang22}.

%%%%%%%%%%%%%%%%%%%%%%%%%%%
\section{Bifurcations of tangent points and their loops}

\setcounter{equation}{0}
\setcounter{lm}{0}
\setcounter{thm}{0}
\setcounter{rmk}{0}
\setcounter{df}{0}
\setcounter{cor}{0}

In this section, we consider system~\eqref{pws1} with an isolated tangent point of multiplicity $(m^+,m^-)$ at the origin $O:(0,0)$ and it is written as
\begin{equation}
\begin{aligned}
        \left( \begin{array}{c}
        \dot{x}\\
        \dot{y}\\
    \end{array} \right) =\left\{
    \begin{aligned}
&       \left( \begin{array}{c}
            f^+(x,y)\\
            \phi^+(x,y)x^{m^+}+\Upsilon^+(x,y)y
        \end{array} \right)   &&\mathrm{if}~(x,y)\in\Sigma^+,\\
&       \left( \begin{array}{c}
            f^-(x,y)\\
            \phi^-(x,y)x^{m^-}+\Upsilon^-(x,y)y
        \end{array} \right)   && \mathrm{if}~(x,y)\in\Sigma^-,
    \end{aligned} \right.
    \end{aligned}
    \label{pws2}
\end{equation}
where $m^++m^-\ge 1$, $f^\pm(0,0)\ne 0, \phi^\pm(0,0)\ne 0$ and $\partial^{m^\pm}\Upsilon^\pm/\partial x^{m^\pm}\equiv 0$ if $m^\pm\ge 1$.

\subsection{Bifurcations of tangent points}

\setcounter{equation}{0}
\setcounter{lm}{0}
\setcounter{thm}{0}
\setcounter{rmk}{0}
\setcounter{df}{0}
\setcounter{cor}{0}

In this subsection, we discuss bifurcations of a tangent point of multiplicity $(m^+,m^-)$.

\begin{thm}
If $O$ is a degenerate tangent point of system~\eqref{pws2}, i.e., $m^++m^-\ge 2$, then
\begin{enumerate}
\item[{\rm (a)}] under perturbations $O$ breaks into at most $m^++m^-$ tangent points and their multiplicities $(m^+_i,m^-_i)$ ($i=1,...,\ell$) satisfy
           \begin{eqnarray}
           \sum_{i=1}^{\ell}m^+_i\le m^+,~~\sum_{i=1}^{\ell}m^-_i\le m^-
           \label{maxi-sum}
           \end{eqnarray}
when exactly $\ell\in\left[1,m^++m^-\right]$ tangent points appear.
\item[{\rm (b)}] $\left(m^+_i,m^-_i\right)\in\left\{(0,1),~(1,0)\right\}$ for all $i=1,...,\ell$ and each subsystem has alternating invisible and visible tangent points bifurcated from $O$ if $\ell=m^++m^-$.
\item[{\rm (c)}] For given $\ell\in\left[1,m^++m^-\right]$ there exist perturbations such that exactly $\ell$ tangent points are bifurcated from $O$ and these two ``$=$'' in \eqref{maxi-sum} hold.
\end{enumerate}
\label{thm1}
\end{thm}

Theorem~\ref{thm1} discovers a relation between number of tangent points bifurcating from $O$
(such tangent points are usually called {\it bifurcating tangent points} from $O$)
and degeneration of tangent point $O$, i.e., tangency multiplicities.
Let $\Gamma^\pm$ be orbit segments of perturbations of system~\eqref{pws2}
in $\overline{\Sigma^\pm}$ respectively and $N(\Gamma^\pm)$ denote the numbers of bifurcating tangent points of $O$ on $\Gamma^{\pm}$. Further, we call $\Gamma^+$ (resp. $\Gamma^-$) a {\it bifurcating tangent orbit} if $N(\Gamma^+)\ge 1$ (resp. $N(\Gamma^-)\ge 1$).

By Theorem~\ref{thm1}, for the upper subsystem of one perturbation system the number of
bifurcating tangent points except the invisible type is at most $(m^+-1)/2$
(resp. $\lfloor(m^++1)/2\rfloor$) in the case that
$O$ of system~\eqref{pws2} is invisible (resp. visible, left or right),
which implies that any orbit in $\Sigma^+$ is tangent with $\Sigma$ at most $(m^+-1)/2$
(resp. $\lfloor(m^++1)/2\rfloor$) times. Here $\lfloor x\rfloor$ means the maximum integer no greater than $x$.
Moreover, all these $(m^+-1)/2$
(resp. $\lfloor(m^++1)/2\rfloor$) bifurcating tangent points may be of visible type by Theorem~\ref{thm1}(b).
On the other hand, there are at most $(m^+-1)/2$ (resp.
$\lfloor(m^++1)/2\rfloor$) distinct orbits passing through one visible bifurcating tangent point. It is natural to obtain that if $O$ is invisible (resp. visible, left or right) for system~\eqref{pws2} and $m^+>1$ (resp. $m^+\ge1$), then for each bifurcating tangent orbit $\Gamma^+$
\begin{equation*}
N(\Gamma^+)\in\left\{0,...,(m^+-1)/2\right\}~\left({\rm resp.}~N(\Gamma^+)\in\left\{0,...,\lfloor(m^++1)/2\rfloor\right\}\right)
\end{equation*}
and, further, for each $\ell\in\left\{1,...,(m^+-1)/2\right\}$ (resp. $\ell\in\left\{1,...,\lfloor(m^++1)/2\rfloor\right\}$) there are at most $\lfloor(m^+-1)/(2\ell)\rfloor$ (resp. $\lfloor(m^++1)/(2\ell)\rfloor$) bifurcating tangent orbits passing through $\ell$ bifurcating tangent points in $\Sigma^+$.

To determine the number of bifurcating tangent orbits passing through a given number of bifurcating tangent points,
in the following we consider system~\eqref{pws2} with functions $\Upsilon^\pm\equiv0$, i.e.,
\begin{equation}
\begin{aligned}
        \left( \begin{array}{c}
        \dot{x}\\
        \dot{y}\\
    \end{array} \right) =\left\{
    \begin{aligned}
&       \left( \begin{array}{c}
            f^+(x,y)\\
            \phi^+(x,y)x^{m^+}
        \end{array} \right)   &&\mathrm{if}~(x,y)\in\Sigma^+,\\
&       \left( \begin{array}{c}
            f^-(x,y)\\
            \phi^-(x,y)x^{m^-}
        \end{array} \right)   && \mathrm{if}~(x,y)\in\Sigma^-,
    \end{aligned} \right.
\end{aligned}
    \label{pws3}
\end{equation}
and its unfolding in form
\begin{equation}
\begin{aligned}
        \left( \begin{array}{c}
        \dot{x}\\
        \dot{y}\\
    \end{array} \right) =\left\{
    \begin{aligned}
&       \left( \begin{array}{c}
            f^+(x,y+\psi^+)\\
            \phi^+(x,y+\psi^+)\prod\limits_{i=1}^{m^+}(x-\lambda^+_i)-f^+(x,y+\psi^+)\dot\psi^+
        \end{array} \right)   &&\mathrm{if}~(x,y)\in\Sigma^+,\\
&       \left( \begin{array}{c}
            f^-(x,y+\psi^-)\\
            \phi^-(x,y+\psi^-)\prod \limits_{i=1}^{m^-}(x-\lambda^-_i)-f^-(x,y+\psi^-)\dot\psi^-
        \end{array} \right)   && \mathrm{if}~(x,y)\in\Sigma^-,
    \end{aligned} \right.
\end{aligned}
    \label{pws4}
\end{equation}
where $\psi^{\pm}:=\psi^{\pm}\left(x,{\boldsymbol k}^{\pm}\right)$, $\dot\psi^{\pm}:=d\psi^{\pm}/dx$,
$\left({\boldsymbol \lambda}^{\pm},{\boldsymbol k}^{\pm}\right)\in\mathbb{R}^{n^\pm}$ and $\psi^{\pm}\left(x,{\boldsymbol 0}\right)\equiv 0$.
For integrity, we mark that $\prod_{i=1}^{m^+}(x-\lambda^+_i)$ (resp. $\prod_{i=1}^{m^-}(x-\lambda^-_i)$) denotes $1$ if $m^+=0$ (resp. $m^-=0$) in this paper.

\begin{thm} Assume that tangent point $O$ is invisible (resp. visible, left or right) for the upper subsystem of \eqref{pws3} and $m^+>1$ (resp. $m^+\ge 1$). Then for each $\ell\in\left\{1,...,(m^+-1)/2\right\}$ (resp. $\ell\in\left\{1,...,\lfloor(m^++1)/2\rfloor\right\}$) there exists ${\boldsymbol \lambda}^+$ and $\psi^+\left(x,{\boldsymbol k}^+\right)$ such that the upper subsystem of \eqref{pws4} has $\lfloor(m^+-1)/(2\ell)\rfloor$ (resp. $\lfloor(m^++1)/(2\ell)\rfloor$) bifurcating tangent orbits passing through $\ell$ visible bifurcating tangent points.
\label{thm2}
\end{thm}

Clearly, there is no bifurcating tangent orbits if $m^+=1$ in the case that tangent point $O$ is invisible for the upper subsystem of \eqref{pws3}. Moreover, it is not hard to prove the similar result for $\Gamma^-$ as $\Gamma^+$ given in Theorem~\ref{thm2} and we omit its statement.

Theorem~\ref{thm1} is a generalization on the results of bifurcating tangent points with respect to number
given in previous publications (see, e.g., \cite{Fang21,Teixeira11,Kuznetsov03,Li20,Han13})
from $m^{\pm}\le2$ to general $m^{\pm}$.
On the numbers of bifurcating tangent orbits and bifurcating tangent points on them, Theorem~\ref{thm2} holds naturally for $m^{\pm}\le2$ but are new for general $m^{\pm}$.

\subsection{Bifurcations of nonsliding loops connecting tangent points}

In this subsection, we investigate the bifurcations of nonsliding loops connecting tangent points defined
in Definition~\ref{df-nonsli}. Assume that
system~\eqref{pws3} has a nonsliding loop $L_*^{cri}\in \mathcal L^{cri}$ satisfying (H) as follows.
\begin{enumerate}
	   \item[(H)] $L_{*}^{cri}$ is clockwise and intersects $\Sigma$ at exactly two points $P: (P_x, 0)\in \Sigma_c$ and $O$, where $\mathcal L^{cri}$ is defined in \eqref{class-Tan} and $P_x<0$.
\end{enumerate}

By the visibility (visible, left, right) of tangent point $O$, under $(x,y,t)\to (x,-y,-t)$ there are totally $14$ types of $L_*^{cri}$ satisfying (H) as shown in Figure~\ref{Fig-Loop}, where $V_l$-$R$ means that tangent point $O$ is visible for the upper subsystem, right for the lower one and $L_*^{cri}$
passes through $O$ along the left segment in $\Sigma^+$. Others in Figure~\ref{Fig-Loop} have similar meanings.
\begin{figure}[htp]
\centering
\subfigure[$V_l$-$V_l$]
 {
  \scalebox{0.34}[0.34]{
   \includegraphics{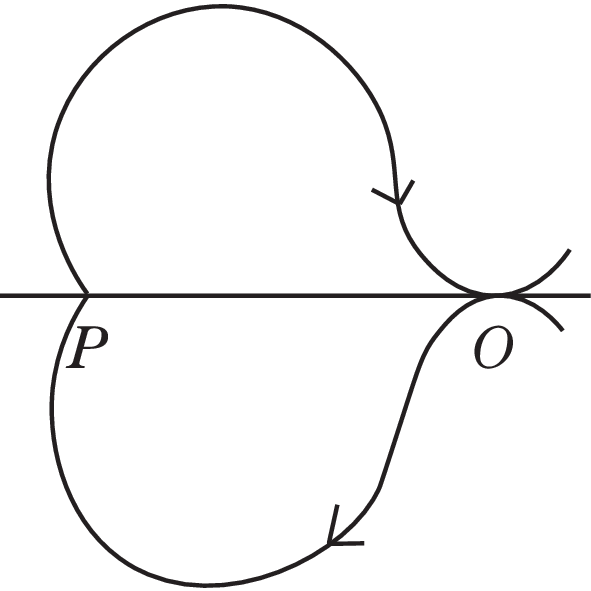}}}
\subfigure[$V_l$-$V_r$]
 {
  \scalebox{0.34}[0.34]{
   \includegraphics{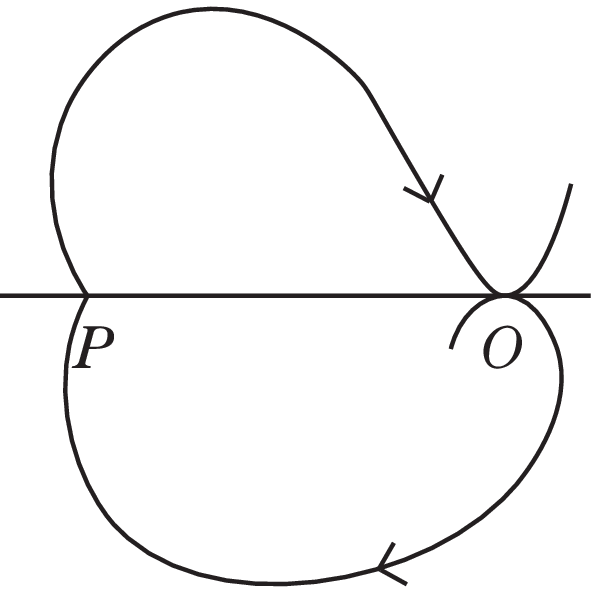}}}
\subfigure[$V_l$-$L$]
 {
  \scalebox{0.34}[0.34]{
   \includegraphics{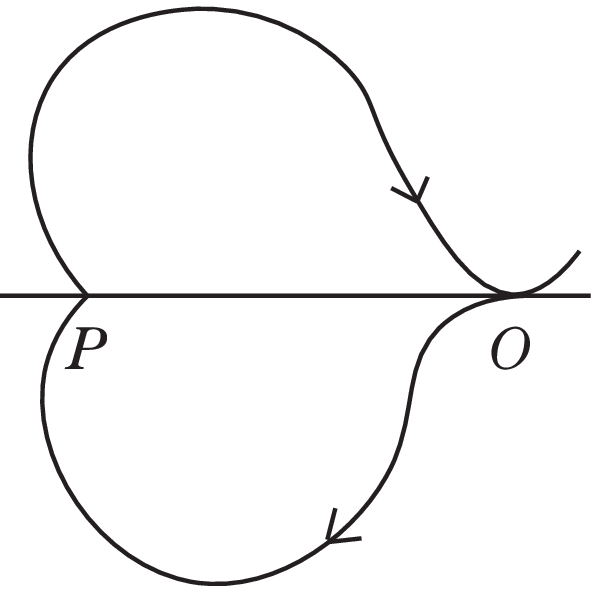}}}
\subfigure[$V_l$-$R$]
 {
  \scalebox{0.34}[0.34]{
   \includegraphics{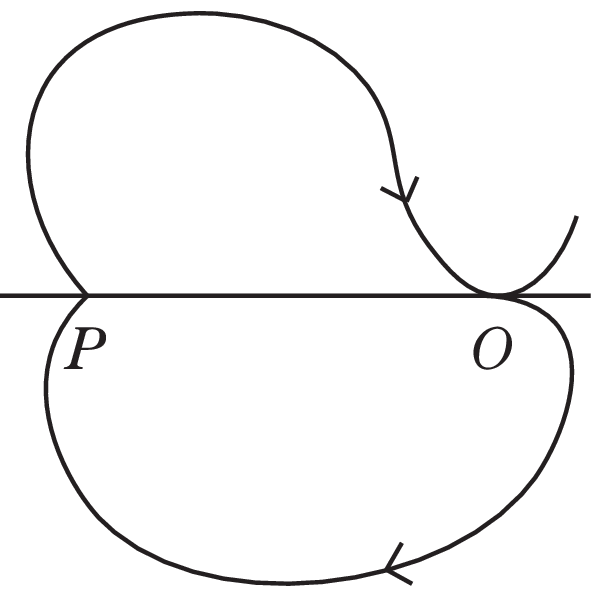}}}
\subfigure[$V_l$-$N$]
 {
  \scalebox{0.34}[0.34]{
   \includegraphics{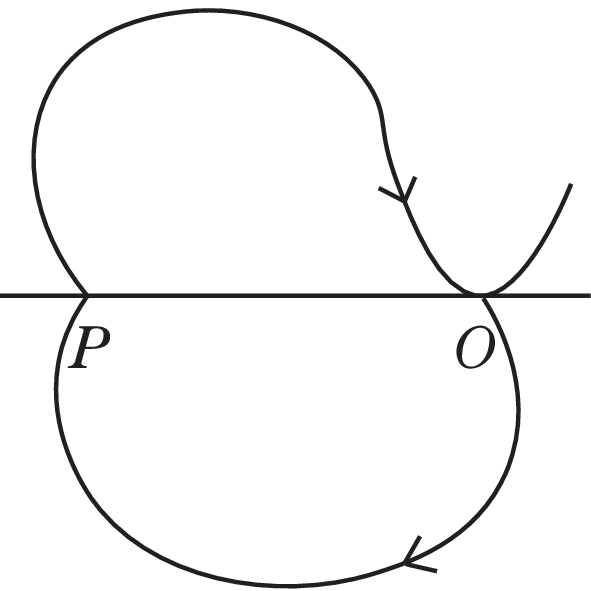}}}
\subfigure[$V_r$-$V_r$]
 {
  \scalebox{0.34}[0.34]{
   \includegraphics{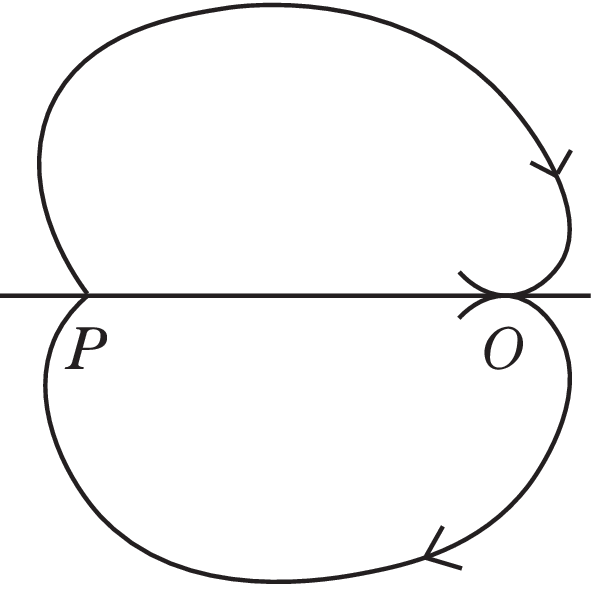}}}
\subfigure[$V_r$-$L$]
 {
  \scalebox{0.34}[0.34]{
   \includegraphics{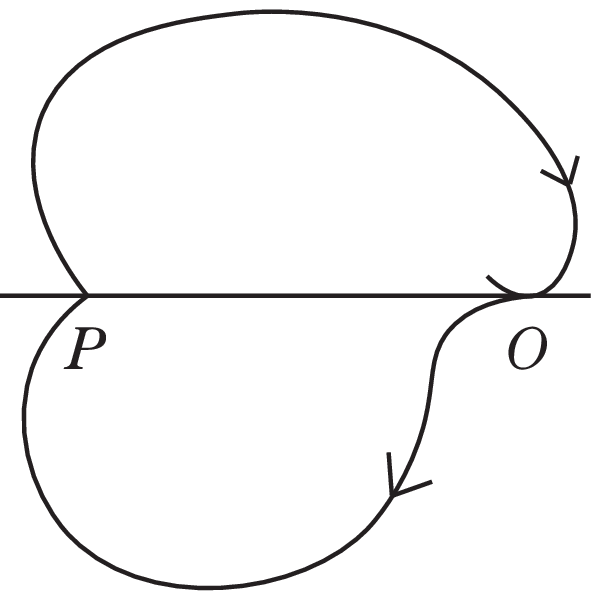}}}
\subfigure[$V_r$-$R$]
 {
  \scalebox{0.34}[0.34]{
   \includegraphics{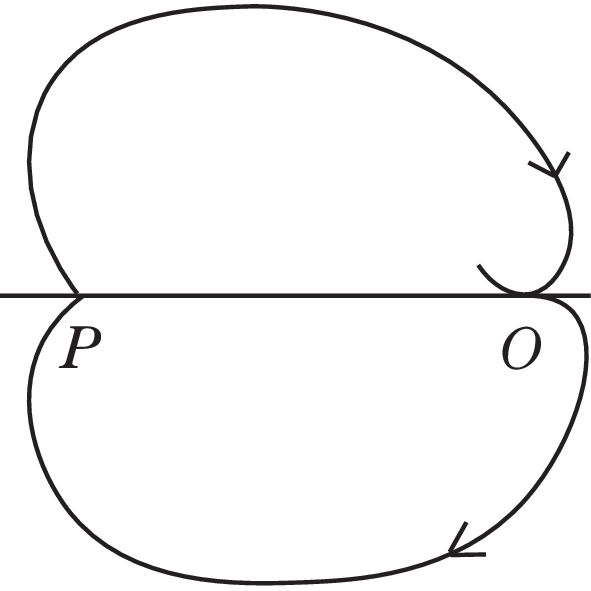}}}
\subfigure[$V_r$-$N$]
 {
  \scalebox{0.34}[0.34]{
   \includegraphics{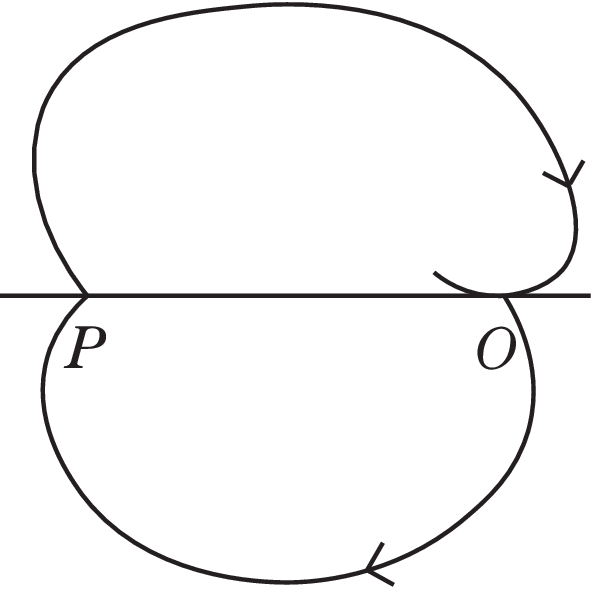}}}
\subfigure[$L$-$L$]
 {
  \scalebox{0.34}[0.34]{
   \includegraphics{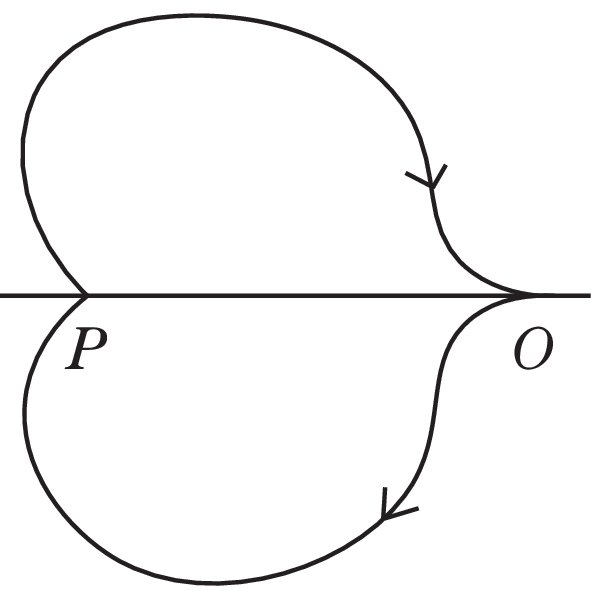}}}
\subfigure[$L$-$R$]
 {
  \scalebox{0.34}[0.34]{
   \includegraphics{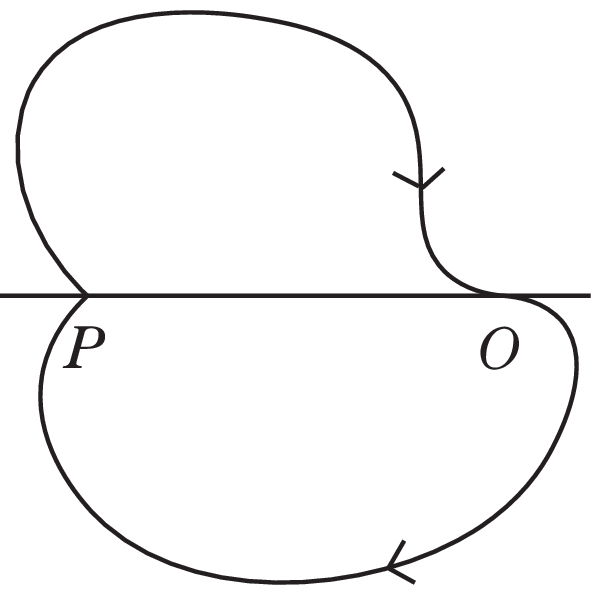}}}
\subfigure[$L$-$N$]
 {
  \scalebox{0.34}[0.34]{
   \includegraphics{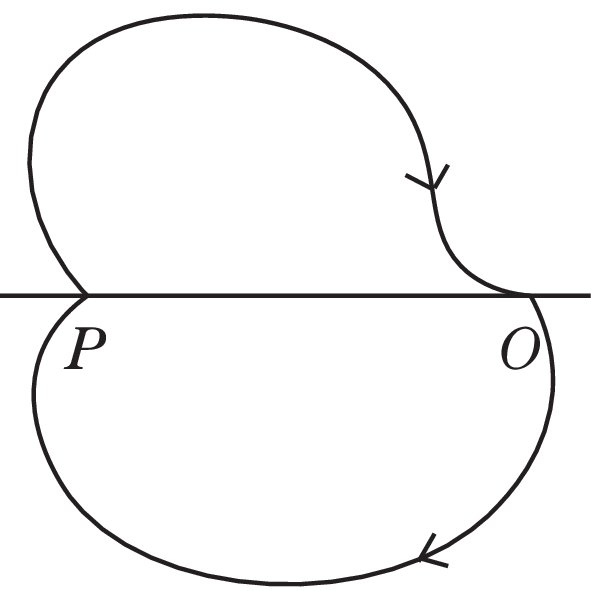}}}
\subfigure[$R$-$R$]
 {
  \scalebox{0.34}[0.34]{
   \includegraphics{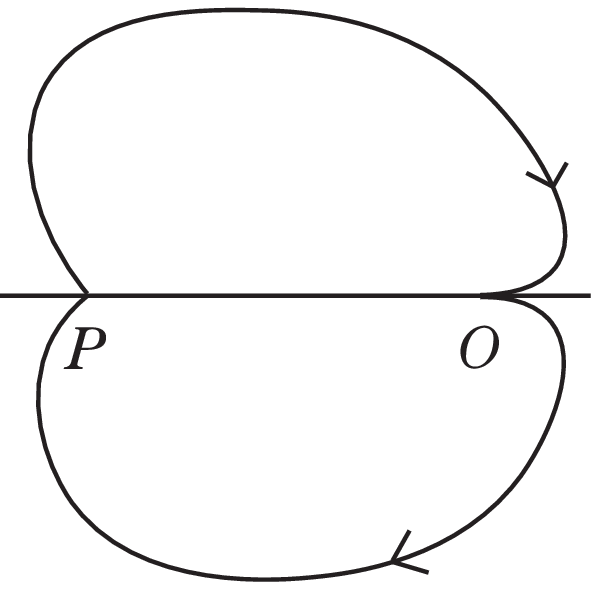}}}
\subfigure[$R$-$N$]
 {
  \scalebox{0.34}[0.34]{
   \includegraphics{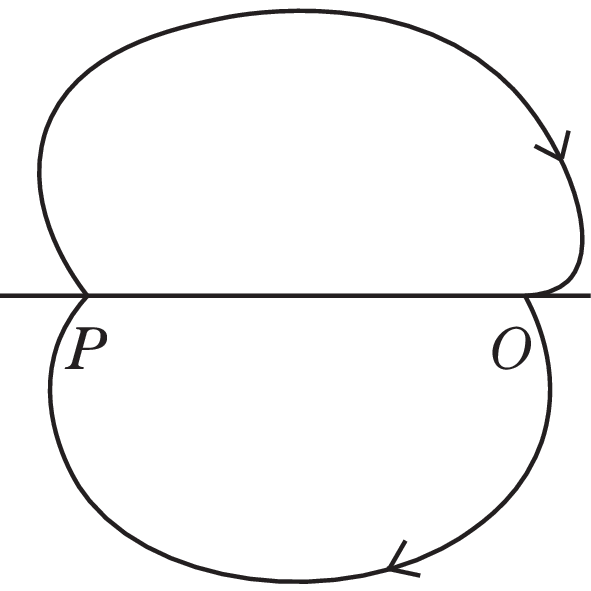}}}
\caption{ types of loop $L_*^{cri}$ }
\label{Fig-Loop}
\end{figure}

In the following, we focus on the existence and number of bifurcating $L_c,~L_s,~L^{gra}$, $L^{cro},~L^{cri}$ of $L_{*}^{cri}$, i.e., loops bifurcating from $L_{*}^{cri}$. Since these bifurcating loops appear in a small neighborhood of $L_{*}^{cri}$, there is no standard equilibria in such neighborhood and further, there is no bifurcating $L^{gra}$. In the following theorem, we give results on the numbers of bifurcating tangent points on bifurcating $L^{cro}$ and $L^{cri}$.

\begin{thm}
Assume that $L_{*}^{cri}$ is a nonsliding loop of system~\eqref{pws3} satisfying {\rm (H)}.
\begin{enumerate}
   \item[{\rm (a)}] There is no bifurcating $L^{cro}$ if $m^*:={\rm max}\left\{m^\pm\right\}=1$ and, otherwise, for each $\ell\in \left\{1,...,\lfloor m^*/2\rfloor\right\}$ there exist ${\boldsymbol \lambda}^\pm$ and $\psi^{\pm}\left(x,{\boldsymbol k}^\pm\right)$ such that system~\eqref{pws4} has a bifurcating $L^{cro}(\ell)$ with exactly $\ell$ bifurcating tangent points.
   \item[{\rm (b)}] For each $\ell\in \left\{1,...,\lfloor (m^*+1)/2 \rfloor\right\}$ there exist ${\boldsymbol \lambda}^\pm$ and $\psi^{\pm}\left(x,{\boldsymbol k}^\pm\right)$ such that system~\eqref{pws4} has a bifurcating $L^{cri}(\ell)$ with exactly $\ell$ bifurcating tangent points.
\end{enumerate}
\label{thm3}
\end{thm}

The nonexistence of bifurcating $L^{cro}$ can be obtained in \cite{Ponce15,Han13} for $(m^+, m^-)=(1,0)$ and in \cite{Novaes18,Huang22} for $(m^+, m^-)=(1, 1)$. Theorem~\ref{thm3} is a generalization from $m^*=1$ to general $m^*$. For the number of bifurcating $L^{cro}$ and $L^{cri}$, we give the following results for $L^{cro}(1)$ and $L^{cri}(1)$, i.e., nonsliding loops with exactly $1$ bifurcating tangent point.

\begin{thm}
Assume that $m^{\pm}\ge 5$ and $L^{cri}_*$ is a nonsliding loop of system~\eqref{pws3} satisfying {\rm (H)}. Then for each $\ell\in\left\{0,...,\lfloor(m^*-1)/2\rfloor\right\}$ there exist ${\boldsymbol \lambda}^\pm$ and $\psi^{\pm}\left(x,{\boldsymbol k}^\pm\right)$ such that
\begin{equation*}
\left(\beta^{cro}, \beta^{cri}\right)=\left(\lfloor(m^*-1)/2\rfloor-\ell,\ell+1\right),
\end{equation*}
where $\beta^{cro},~\beta^{cri}$ are the numbers of bifurcating $L^{cro}(1), L^{cri}(1)$ of system~\eqref{pws4} respectively and $m^*:={\rm max}\{m^{\pm}\}.$
\label{thm4}
\end{thm}

In Theorem~\ref{thm4} the relation between the numbers of $L^{cro}(1), L^{cri}(1)$ and degeneration of tangent point $O$ is obtained for $m^{\pm}\ge 5$ but not for ${\rm min}\left\{m^{\pm}\right\}<5$. The main reason is that only $m^{\pm}\ge 5$ can guarantee the analysis is strictly under the framework of bifurcations, i.e., the vector field of \eqref{pws4} lies in a small neighborhood of the vector field of \eqref{pws3}.

In the following theorem, we provide the numbers of bifurcating $L_c$ and $L_s$.

\begin{thm} Assume that $L^{cri}_*$ is a nonsliding loop of system~\eqref{pws3} satisfying {\rm (H)}. Then there exist ${\boldsymbol \lambda}^{\pm}$ and $\psi^{\pm}\left(x,{\boldsymbol k}^\pm\right)$ in system~\eqref{pws4} such that $\beta_c+\beta_s\ge 1$. In the case that $m^{\pm}\ge 5$, for each $\ell\in\left\{0,...,\lfloor(m^*+1)/2\rfloor\right\}$ there exist ${\boldsymbol \lambda}^{\pm}$ and $\psi^{\pm}\left(x,{\boldsymbol k}^\pm\right)$ such that
   \begin{equation*}
   \beta_c\ge m^*-\lfloor(m^*+1)/2\rfloor+\ell,~\beta_s=\lfloor(m^*+1)/2\rfloor-\ell.
   \end{equation*}
Here $\beta_c, \beta_s$ denote the numbers of bifurcating $L_c, L_s$ of system~\eqref{pws4} respectively and $m^*:={\rm max}\{m^{\pm}\}$.
\label{thm5}
\end{thm}

Theorem~\ref{thm5} mainly gives a lower bound of $\beta_c+\beta_s$ and the coexistence
of  bifurcating $L_c, L_s$ for system~\eqref{pws4} with $m^{\pm}\ge 5$. Clearly, $\beta_c+\beta_s\ge m^*$ no matter $m^*$ is odd or even.
For the case $(m^+,m^-)=(1,0), (0,1)$, the uniqueness of $L_c, L_s$ bifurcated from $L^{cri}_*$ is proved in \cite{Ponce15,Kuznetsov03,Han13},
i.e., the sum of the numbers of bifurcating $L_c, L_s$ is $1$. For the case $(m^+,m^-)=(1,1)$,
it is proved in \cite{Han13,Novaes18,Huang22} that the sum of the numbers of bifurcating $L_c, L_s$ from $L^{cri}_*$ is at least $2$.

%%%%%%%%%%%%%%%%%%%%%%%%%%%%%%%%%%%%%%%%%%%%%%%%%%%%%%%%%%%
\section{Explicit unfolding of system~\eqref{pws3}}
\setcounter{equation}{0}
\setcounter{lm}{0}
\setcounter{thm}{0}
\setcounter{rmk}{0}
\setcounter{df}{0}
\setcounter{cor}{0}

In this section, some precise statements concluding conception of bifurcations
for general system~\eqref{pws1} and definition of $\psi^{\pm}(x,{\boldsymbol k}^\pm)$ in unfolding~\eqref{pws4} are given.
Thus, we provide some explicit unfolding systems for \eqref{pws3}, which are used
in the next section to prove our main results.

Let $\chi(x,y), \widetilde\chi(x,y)$ be two piecewise-smooth vector fields of form \eqref{pws1} in ${\cal U}$ with switching manifold $y=0$.
Their upper vector fields and lower ones are functions
\begin{equation*}
\begin{aligned}
&\chi^\pm(x,y):=\left(\chi^\pm_1(x,y),\chi^\pm_2(x,y)\right)^\top\in C^\infty\left(\mathbb{R}^2,\mathbb{R}^2\right),\\
&\widetilde\chi^\pm(x,y):=\left(\widetilde\chi^\pm_1(x,y),\widetilde\chi^\pm_2(x,y)\right)^\top\in C^\infty\left(\mathbb{R}^2,\mathbb{R}^2\right).
\end{aligned}
\end{equation*}
Define the distance of piecewise-smooth systems having vector fields $\chi(x,y), \widetilde\chi(x,y)$ as
\begin{eqnarray}
\rho\left(\chi,\widetilde\chi\right):=d\left(\chi^+_1,\widetilde\chi^+_1\right)+d\left(\chi^+_2,\widetilde\chi^+_2\right)+d\left(\chi^-_1,\widetilde\chi^-_1\right)+d\left(\chi^-_2,\widetilde\chi^-_2\right),
\label{dist}
\end{eqnarray}
where
\begin{equation*}
d\left(X, Y\right):=\max_{(x,y)\in\bar{\cal U}}\left\{\sum_{i_1+i_2=0}^{1}\left|\frac{\partial^{i_1+i_2} (X-Y)}{\partial x^{i_1}\partial y^{i_2}}\right|\right\}
\end{equation*}
for $X(x,y),Y(x,y)\in C^\infty\left(\mathbb{R}^2,\mathbb{R}\right)$.

The appearance of nonequivalent phase portraits under the variation of vector fields of PWS system~\eqref{pws1} is called bifurcation. The ``nonequivalent'' is the opposite conception of {\it $\Sigma$-equivalent} defined in \cite{Teixeira11} and the variation is required to be small under the meaning of distance $\rho$. In the following we define $\psi^{\pm}(x,{\boldsymbol k}^{\pm})$ appearing in system~\eqref{pws4} and give a sufficient condition such that the variation of vector fields is sufficiently small as $\left({\boldsymbol \lambda}^{\pm},{\boldsymbol k}^{\pm}\right)$ are sufficiently small.

Consider $C^\infty$ cutoff functions (see, e.g., \cite{Lee18})
\begin{equation*}
h^*(x,r_1,r_2):=\left\{
\begin{aligned}
&0,&&x\in (-\infty,r_1], \\
&\frac{1}{1+e^{\eta(x)}},&&x\in(r_1,r_2), \\
&1,&&x\in[r_2,+\infty),
\end{aligned}
\right.~~
h_*(x,r_1,r_2):=\left\{
\begin{aligned}
&1,&&x\in(-\infty,r_1], \\
&\frac{e^{\eta(x)}}{1+e^{\eta(x)}},&&x\in(r_1,r_2), \\
&0,&&x\in[r_2,+\infty),
\end{aligned}
\right.
\end{equation*}
where $r_1<r_2$ and
\begin{equation*}
\eta(x):=\frac{1}{x-r_1}+\frac{1}{x-r_2}.
\end{equation*}
Let
\begin{equation*}
{\cal K}:=\bigcap_{i=1}^{2d}\left\{{\boldsymbol k}\in{\mathbb R}^{3d+1}:~k_i< k_{i+1}\right\},
\end{equation*}
where integer $d>0$. For $x\in\mathbb{R}$ and any integer $d>0$, we define
\begin{eqnarray}
\psi(x,{\boldsymbol k}):=
\left\{
\begin{aligned}
&\psi^*(x,{\boldsymbol k}), &&{\rm when}~{\boldsymbol k}\in{\cal K},\\
&k_{2d+2}h^*(x,r_1,r_2),~~~~~~&&{\rm when}~{\boldsymbol k}\in\mathbb{R}^{3d+1}\setminus {\cal K},
\end{aligned}
\right.
\label{psidef}
\end{eqnarray}
where
\begin{equation*}
\psi^*(x,{\boldsymbol k}):=
\left\{
\begin{aligned}
& k_{2d+1+i}h^*\left(x,k_{2i-1},k_{2i}\right),~~~~&&x\in\left(k_{2i-1},k_{2i}\right],\\
& k_{2d+1+i}h_*\left(x,k_{2i},k_{2i+1}\right),~~~~&&x\in\left(k_{2i},k_{2i+1}\right],\\
& 0,&&{\rm otherwise}
\end{aligned}
\right.
\end{equation*}
and $i=1,...,d$. It is not hard to check that
$\psi\left(x,{\boldsymbol k}\right)$ is $C^{\infty}$ in $x$ by analyzing $h^*$ and $h^*$.
Moreover, derivatives of $\psi^*\left(x,{\boldsymbol k}\right)$ of any orders with respect to $x$ are all zero at $k_i$ ($i=1,...,2d+1$).
Some examples of $\psi(x,{\boldsymbol k})$ are shown in Figure~\ref{Fig-psi}.

\begin{figure}[htp]
\centering
\subfigure[ ${\boldsymbol k}\in{\cal K}$ and $d=2$ ]
 {
  \scalebox{0.43}[0.43]{
   \includegraphics{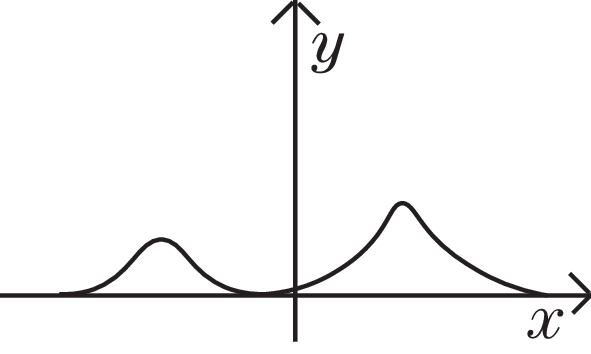}}}~~~~
\subfigure[ ${\boldsymbol k}\in\mathbb{R}^{3d+1}\setminus {\cal K}$ and $k_{2d+2}>0$]
 {
  \scalebox{0.43}[0.43]{
   \includegraphics{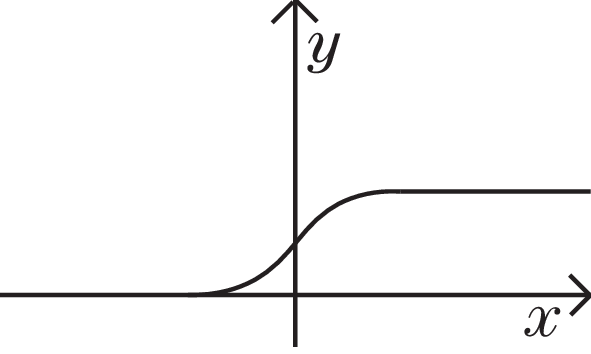}}}
\caption{ examples of $\psi(x;{\boldsymbol k})$}
\label{Fig-psi}
\end{figure}

\begin{prop}
If either ${\boldsymbol k}\in \mathbb{R}^{3d+1}\setminus {\cal K}$ or ${\boldsymbol k}\in {\cal K}$ satisfying
\begin{equation}
k_{i+2d+1}=o\left(\left(k_{2i}-k_{2i-1}\right)^4\right),~~~k_{2i+1}-k_{2i}=k_{2i}-k_{2i-1},~~{\rm~for~}i=1,...,d,
\label{K}
\end{equation}
then
$$
\psi(x,{\boldsymbol k})\rightrightarrows0,~~~\dot\psi(x,{\boldsymbol k})\rightrightarrows0,~~~\ddot\psi(x,{\boldsymbol k})\rightrightarrows0
$$
as ${\boldsymbol k}\to{\boldsymbol 0}$, where $\dot \psi, \ddot\psi$ mean the first and second derivatives with respect
to $x$ respectively and ``$\rightrightarrows$'' denotes uniform convergence.
\label{proppsi}
\end{prop}

\begin{proof}
Consider $\psi\left(x,{\boldsymbol k}\right)$ with ${\boldsymbol k}\in{\mathbb R}^{3d+1}\setminus {\cal K}$ firstly.
We get $\psi(x,{\boldsymbol k})\le |k_{2d+2}|$, which implies that
$\psi\left(x,{\boldsymbol k}\right)\rightrightarrows 0$ as ${\boldsymbol k}\to {\boldsymbol 0}$.
It is not hard to check that $h^*(x,r_1,r_2)$ has bounded derivative, which is independent from
${\boldsymbol k}$. Thus, there exists
$M_1>0$ independent from ${\boldsymbol k}$
such that $\dot\psi(x,{\boldsymbol k})\le \left|k_{2d+2}\right|M_1$. This implies that
$\dot\psi(x,{\boldsymbol k})\rightrightarrows 0$ as ${\boldsymbol k}\to {\boldsymbol 0}$.
Similarly, by the definition of $h^*$ we can prove that
$\ddot\psi(x,{\boldsymbol k})\rightrightarrows 0$ as ${\boldsymbol k}\to {\boldsymbol 0}$.

Now consider $\psi\left(x,{\boldsymbol k}\right)$ with ${\boldsymbol k}\in{\cal K}$ satisfying~\eqref{K}.
By \eqref{psidef} we get $\psi\left(x,{\boldsymbol k}\right) \le \max\left\{\left|k_{2d+2}\right|,...,\left|k_{3d+1}\right|\right\}$, which implies
$\psi\left(x,{\boldsymbol k}\right)\rightrightarrows0$ as ${\boldsymbol k}\to {\boldsymbol 0}$.
Straight computation shows that for $x\in (k_1,k_2)$
$$
\dot\psi(x,{\boldsymbol k})=k_{2d+2}\nu_2(x)\frac{e^{\nu_1(x)}}{\left(1+e^{\nu_1(x)}\right)^2},~~~~~
\ddot\psi(x,{\boldsymbol k})=k_{2d+2}R_1(x)\frac{e^{\nu_1(x)}}{\left(1+e^{\nu_1(x)}\right)^2}
$$
and $\dot\psi(k_{1,2},{\boldsymbol k})=\ddot\psi(k_{1,2},{\boldsymbol k})=0$,
where
$R_1(x)=-\nu_3(x)-\nu_2^2(x)+2\nu_2^2(x){e^{\nu_1(x)}}/({1+e^{\nu_1(x)}})$,
\begin{eqnarray*}
\nu_i(x):=\frac{1}{\left(x-k_1\right)^i}+\frac{1}{\left(x-k_2\right)^i},~~~~
\nu_3(x):=\frac{2}{\left(x-k_1\right)^3}+\frac{2}{\left(x-k_2\right)^3}
\end{eqnarray*}
for $i=1,2$.
It is not hard to check that $x^*\in(k_1,k_2)$ satisfies $\ddot\psi(x^*,{\boldsymbol k})=0$ if and only if
\begin{eqnarray}
\nu_3(x^*)+\nu_2^2(x^*)+\left(\nu_3(x^*)-\nu_2^2(x^*)\right)e^{\nu_1(x^*)}=0.
\label{xstar}
\end{eqnarray}
Letting $\delta_1=k_2-k_1$ and $x^*=k_1+\theta^*\delta_1$ for $\theta^*\in(0,1)$, from \eqref{xstar} we get that
$(\delta_1,\theta^*)$ satisfies
$$
\frac{2\theta^*-1}{\theta^*(\theta^*-1)}-\delta_1\ln \frac{\mu_2(\theta^*)+\delta_1\mu_1(\theta^*)}{\mu_2(\theta^*)-\delta_1\mu_1(\theta^*)}=0.
$$
By the Implicit Function Theorem, $\theta^*=1/2+O(\delta_1)$. On the other hand, the above equality holds for $\theta^*=1/2$ and
any $\delta_1$. Thus, $\dot\psi(x,{\boldsymbol k})$ has a unique extremum value point $x^*=(k_1+k_2)/2$.
Then, associated with $\dot\psi(k_{1,2},{\boldsymbol k})=0$,
$$
|\dot\psi(x,{\boldsymbol k})|
\le |\dot\psi(x^*,{\boldsymbol k})|= |k_{2d+2}|\nu_2(x^*) = 8|k_{2d+2}|(k_2-k_1)^{-2}
$$
for $x\in [k_1,k_2]$.
By \eqref{K}, $\dot\psi\left(x,{\boldsymbol k}\right)\to 0$ uniformly for $x\in [k_1,k_2]$ as ${\boldsymbol k}\to {\boldsymbol 0}$.
Similarly, this convergence is also uniform for $x\in [k_i,k_{i+1}]$ ($i=2,...,2d$).
Thus, $\dot\psi\left(x,{\boldsymbol k}\right)\rightrightarrows0$ as ${\boldsymbol k}\to {\boldsymbol 0}$.

Similarly, for $x\in [k_1,k_2]$ there exists $M_2>0$ such that
$
|\ddot\psi(x,{\boldsymbol k})| \le M_2 |k_{2d+2}|(k_2-k_1)^{-4}=O(k_2-k_1).
$ By \eqref{K}, $\ddot\psi\left(x,{\boldsymbol k}\right)\to 0$ uniformly for $x\in [k_1,k_2]$ as ${\boldsymbol k}\to {\boldsymbol 0}$.
Further, it can be proved similarly that this uniform convergence holds
for $x\in[k_i,k_{i+1}]$ ($i=2,...,2d$) and the proof is finished.
\end{proof}

\begin{prop}
If $\psi^{\pm}(x,{\boldsymbol k}^{\pm})$ are of form $\psi(x,{\boldsymbol k})$ defined in \eqref{psidef} and
either ${\boldsymbol k}^{\pm}\in \mathbb{R}^{3d+1}\setminus {\cal K}$ or ${\boldsymbol k}^{\pm}\in {\cal K}$ satisfying
\eqref{K},
then for any $\epsilon>0$ there exists $\delta>0$ such that
\begin{equation*}
\rho\left(Z,\widetilde Z\right)<\epsilon,~~~~\forall |{\boldsymbol \lambda}^{\pm}|, |{\boldsymbol k}^{\pm}|<\delta,
\end{equation*}
where $Z(x,y),\widetilde Z\left(x,y\right)$ are piecewise-smooth vector fields of \eqref{pws3}, \eqref{pws4} respectively.
\label{prop1}
\end{prop}

\begin{proof}
By the definition of $\rho$ in \eqref{dist}, we obtain
$\rho\left(Z,\widetilde Z\right)=d_1+d_2+d_3+d_4$,
where
\begin{equation*}
\begin{aligned}
&d_1=\max_{(x,y)\in \bar{\cal U}}\left\{\left|Z^+_1-\widetilde Z^+_1\right|+\left|Z^+_{1x}-\widetilde Z^+_{1x}\right|+\left|Z^+_{1y}-\widetilde Z^+_{1y}\right|\right\},\\
&d_2=\max_{(x,y)\in \bar{\cal U}}\left\{\left|Z^+_2-\widetilde Z^+_2\right|+\left|Z^+_{2x}-\widetilde Z^+_{2x}\right|+\left|Z^+_{2y}-\widetilde Z^+_{2y}\right|\right\},\\
&d_3=\max_{(x,y)\in \bar{\cal U}}\left\{\left|Z^-_1-\widetilde Z^-_1\right|+\left|Z^-_{1x}-\widetilde Z^-_{1x}\right|+\left|Z^-_{1y}-\widetilde Z^-_{1y}\right|\right\},\\
&d_4=\max_{(x,y)\in \bar{\cal U}}\left\{\left|Z^-_2-\widetilde Z^-_2\right|+\left|Z^-_{2x}-\widetilde Z^-_{2x}\right|+\left|Z^-_{2y}-\widetilde Z^-_{2y}\right|\right\}.
\end{aligned}
\end{equation*}
Here
$\left(Z^\pm_1(x,y),Z^\pm_2(x,y)\right)^\top$,  $\left(\widetilde Z^\pm_1(x,y),\widetilde Z^\pm_2(x,y)\right)^\top$
are the upper and lower vector fields of systems~\eqref{pws3} and \eqref{pws4}.

For $(x,y)\in {\cal U}$, by Mean Value Theorem we get
\begin{equation}
\left|Z_1^+(x,y)-\widetilde Z^+_1(x,y)\right|
 = \left|f^+(x,y+\psi^+)-f^+(x,y)\right|
 = \left|\psi^+\right|\cdot\left|f^+_y(x,y+\theta_1\psi^+)\right|,
 \label{f+d1-1}
\end{equation}
where $\theta_1\in(0,1)$. Similarly,
\begin{equation}
\begin{aligned}
\left|Z^+_{1x}(x,y)-\widetilde Z^+_{1x}(x,y)\right|
& = \left|f^+_x\left(x,y+\psi^+\right)+f^+_y\left(x,y+\psi^+\right)\dot\psi^+-f^+_x(x,y)\right|\\
& = \left|\psi^+f^+_{xy}\left(x,y+\theta_2\psi^+\right)+f^+_y\left(x,y+\psi^+\right)\dot\psi^+\right|\\
& \le \left|\psi^+\right|\cdot\left|f^+_{xy}\left(x,y+\theta_2\psi^+\right)\right|+\left|\dot\psi^+\right|\cdot\left|f^+_y\left(x,y+\psi^+\right)\right|,
\\
\left|Z^+_{1y}(x,y)-\widetilde Z^+_{1y}(x,y)\right|
& = \left|f^+_y\left(x,y+\psi^+\right)-f^+_y(x,y)\right|\\
& = \left|\psi^+f^+_{yy}\left(x,y+\theta_3\psi^+\right)\right|\\
& = \left|\psi^+\right|\cdot\left|f^+_{yy}\left(x,y+\theta_3\psi^+\right)\right|,
\label{f+d1-2}
\end{aligned}
\end{equation}
where $\theta_2,~\theta_3\in(0,1)$. By the $C^\infty$ smoothness of $f^+$ on $\mathbb{R}^2$,
we have the boundedness of $f^+_y(x,y+\psi^+)$, $f^+_y(x,y+\theta_1\psi^+)$, $f^+_{xy}\left(x,y+\theta_2\psi^+\right)$
and $f^+_{yy}\left(x,y+\theta_3\psi^+\right)$ on $\bar {\cal U}$.
Associated with Proposition~\ref{proppsi}, \eqref{f+d1-1} and \eqref{f+d1-2}, for any $\epsilon>0$
there exists $\delta_1(\epsilon)>0$ such that
$d_1<\epsilon/4$ for all $|{\boldsymbol \lambda}^\pm|,|{\boldsymbol k}^\pm|<\delta_1(\epsilon)$.

Similarly, for any $\epsilon>0$
there exists $\delta_i(\epsilon)>0$ such that
$d_i<\epsilon/4$ for all $|{\boldsymbol \lambda}^\pm|,|{\boldsymbol k}^\pm|<\delta_i(\epsilon)$, $i=2,3,4$.
Thus, $\rho\left(Z,\widetilde Z\right)<\epsilon$ for all
$|{\boldsymbol \lambda}^\pm|,|{\boldsymbol k}^\pm|<\min\{\delta_i(\epsilon):~i=1,...,4\}$.
\end{proof}

Proposition~\ref{prop1} shows that $\widetilde Z\left(x,y\right)$ is at least continuous with respect to ${\boldsymbol k}^\pm$ at ${\boldsymbol k}^\pm={\boldsymbol 0}$, which makes sure that all analysis is strictly under the framework of bifurcation.
On the other hand, in section 1 we introduce functional parameters for perturbations of given PWS systems,
which are used to determine some dynamics of perturbation systems.
In perturbation system~\eqref{pws4}, the function of ${\boldsymbol \lambda}^\pm$ is to
desingularize tangent point $O$, i.e., to decrease the multiplicity $(m^+,m^-)$ of $O$.
Clearly, in perturbation
\begin{equation}
\begin{aligned}
        \left( \begin{array}{c}
        \dot{x}\\
        \dot{y}\\
    \end{array} \right) =\left\{
    \begin{aligned}
&       \left( \begin{array}{c}
            f^+(x,y)\\
            \phi^+(x,y)\prod\limits_{i=1}^{m^+}(x-\lambda^+_i)
        \end{array} \right)   &&\mathrm{if}~(x,y)\in\Sigma^+,\\
&       \left( \begin{array}{c}
            f^-(x,y)\\
            \phi^-(x,y)\prod \limits_{i=1}^{m^-}(x-\lambda^-_i)
        \end{array} \right)   && \mathrm{if}~(x,y)\in\Sigma^-,
    \end{aligned} \right.
\end{aligned}
    \label{tpws}
\end{equation}
tangent point $O$ of multiplicity $(m^+,m^-)$ breaks into several tangent points of lower multiplicities.
\eqref{tpws} can be looked as a transition system from system~\eqref{pws3} to system~\eqref{pws4}.

In the following we show the function of $\psi^\pm(x,{\boldsymbol k}^\pm)$ in perturbation system~\eqref{pws4}.

\begin{prop}
Assume that $\psi^\pm(x,{\boldsymbol k}^\pm)$ are of form $\psi(x,{\boldsymbol k})$ defined in \eqref{psidef} and ${\boldsymbol k}^\pm\in\mathbb{R}^{3d+1}\setminus{\cal K}$ or ${\boldsymbol k}^\pm\in{\cal K}$ satisfying~\eqref{K}.
Then for $(x_0^\pm,y_0^\pm)\in\Sigma^\pm$ there exist $\delta^\pm>0$ such that for $|{\boldsymbol k}^\pm|<\delta^\pm$,
\begin{equation}
\left(\!
\begin{aligned}
&\widetilde\gamma^\pm_1\left(t,x_0^\pm,y_0^\pm\!-\!\psi^\pm\left(x_0^\pm,{\boldsymbol k}^\pm\right)\right)\\
&\widetilde\gamma^\pm_2\left(t,x_0^\pm,y_0^\pm\!-\!\psi^\pm\left(x_0^\pm,{\boldsymbol k}^\pm\right)\right)
\end{aligned}
\!\right)=\!
\left(\!
\begin{aligned}
&\widehat\gamma^\pm_1\left(t,x_0^\pm,y_0^\pm\right)\\
&\widehat\gamma^\pm_2\left(t,x_0^\pm,y_0^\pm\right)
\end{aligned}
\!\right)
\!-
\left(\!
\begin{aligned}
&0\\
&\psi^\pm\left(\widehat \gamma^\pm_1\left(t,x_0^\pm,y_0^\pm\right),{\boldsymbol k}^\pm\right)
\end{aligned}
\!\right)
\label{defor}
\end{equation}
over some intervals $I^\pm$, where
\begin{equation*}
\begin{aligned}
&\widetilde\gamma^\pm\left(t,x^\pm_0,y^\pm_0\right):=\left(\widetilde\gamma^\pm_1\left(t,x_0^\pm,y_0^\pm\right), \widetilde\gamma^\pm_2\left(t,x_0^\pm,y_0^\pm\right)\right)^\top,\\
&\widehat\gamma^\pm\left(t,x^\pm_0,y^\pm_0\right):=\left(\widehat\gamma^\pm_1\left(t,x_0^\pm,y_0^\pm\right), \widehat\gamma^\pm_2\left(t,x_0^\pm,y_0^\pm\right)\right)^\top
\end{aligned}
\end{equation*}
denote solutions of the upper and lower subsystems of \eqref{pws4}, \eqref{tpws} with initial values $(x_0^\pm,y_0^\pm)$.
\label{prop2}
\end{prop}

\begin{proof}
Let $\widehat Z(x,y)$ be the piecewise-smooth vector field of system~\eqref{tpws} and
denote its upper vector field and lower one by
\begin{equation*}
\widehat Z^+(x,y):=\left(\widehat Z^+_1(x,y),\widehat Z^+_2(x,y)\right)^\top,~~~\widehat Z^-(x,y):=\left(\widehat Z^-_1(x,y),\widehat Z^-_2(x,y)\right)^\top.
\end{equation*}
For convenience, we write $\psi^+(x)$ and $\widehat\gamma^+(t)$ to replace $\psi^+\left(x,{\boldsymbol k}^+\right)$ and $\widehat\gamma^+\left(t,x^+_0,y^+_0\right)$ respectively. Clearly,
\begin{equation*}
\widehat\gamma^+_1(t)=x^+_0+\int_{0}^{t}\widehat Z^+_1\left(\widehat\gamma^+_1(s),\widehat\gamma^+_2(s)\right)ds,~~~\widehat\gamma^+_2(t)=y^+_0+\int_{0}^{t}\widehat Z^+_2\left(\widehat\gamma^+_1(s),\widehat\gamma^+_2(s)\right)ds.
\end{equation*}
By $\widehat Z^+_1\left(x,y\right) = \widetilde Z^+_1\left(x,y-\psi^+\left(x\right)\right)$, we get
\begin{equation}
\widehat\gamma^+_1(t)=x^+_0+\int_{0}^{t}\widetilde Z^+_1\left(\widehat\gamma^+_1(s),\widehat\gamma^+_2(s)-\psi^+\left(\widehat\gamma^+_1(s)\right)\right)ds.
\label{eq-in-1}
\end{equation}
Similarly, by $\widehat Z^+_2\left(x,y\right) = \widetilde Z^+_2\left(x,y-\psi^+\left(x\right)\right)+\widetilde Z^+_1\left(x,y-\psi^+\left(x\right)\right)\dot\psi\left(x\right)$,
\begin{equation}
\begin{aligned}
\widehat\gamma^+_2(t)
& = y^+_0+\int_{0}^{t}\widetilde Z^+_2\left(\widehat\gamma^+_1(s),\widehat\gamma^+_2(s)-\psi^+\left(\widehat\gamma^+_1(s)\right)\right)ds\\
& ~~~+\int_{0}^{t}\widetilde Z^+_1\left(\widehat\gamma^+_1(s),\widehat\gamma^+_2(s)-\psi^+\left(\widehat\gamma^+_1(s)\right)\right)\dot\psi^+\left(\widehat\gamma^+_1(s)\right)ds\\
& = y^+_0+\int_{0}^{t}\widetilde Z^+_2\left(\widehat\gamma^+_1(s),\widehat\gamma^+_2(s)-\psi^+\left(\widehat\gamma^+_1(s)\right)\right)ds+\int_{0}^{t}\dot\psi^+\left(\widehat\gamma^+_1(s)\right)d\widehat\gamma^+_1(s)\\
& = y^+_0-\psi^+(x^+_0)+\int_{0}^{t}\widetilde Z^+_2\left(\widehat\gamma^+_1(s),\widehat\gamma^+_2(s)-\psi^+\left(\widehat\gamma^+_1(s)\right)\right)ds+\psi^+\left(\widehat\gamma^+_1(t)\right).
\end{aligned}
\label{eq-in-2}
\end{equation}
By (\ref{eq-in-1}) and (\ref{eq-in-2}), we obtain that
$(\widehat\gamma^+_1\left(t,x^+_0,y^+_0\right), \widehat\gamma^+_2\left(t,x^+_0,y^+_0\right)-\psi^+\left(\widehat\gamma^+_1\left(t,x^+_0,y^+_0\right)\right))^\top$
satisfies the equivalent integral equation of \eqref{pws4}.
Thus, it is the solution of \eqref{pws4} with initial value $\left(x^+_0,y^+_0-\psi^+(x^+_0)\right)^\top$.
Therefore, \eqref{defor} with `+' is proved.
\eqref{defor} with `-' can be proved similarly for the lower subsystems of \eqref{pws4}, \eqref{tpws} and we omit the statements.
\end{proof}

Associated with the property of $\psi^\pm(x,{\boldsymbol k}^\pm)$,
the relation between orbits $\widehat\gamma^\pm(t,x,y)$ and orbits $\widetilde\gamma^\pm(t,x,y)$ given in Proposition~\ref{prop2}, which greatly help us prove our main results in next two sections.

%%%%%%%%%%%%%%%%%%%%%%%%%%%%%%%%%%%%%%%%%%%%%%%%%%%%%%%%%%%
\section{Proofs of Theorems~\ref{thm1} and \ref{thm2}}
\setcounter{equation}{0}
\setcounter{lm}{0}
\setcounter{thm}{0}
\setcounter{rmk}{0}
\setcounter{df}{0}
\setcounter{cor}{0}

In this section we prove the results on bifurcating tangent points and bifurcating tangent orbits
of tangent point $O$ of multiplicity $(m^+, m^-)$, i.e., Theorems~\ref{thm1} and \ref{thm2}.

\begin{proof}[Proof of Theorem~\ref{thm1}] For system~\eqref{pws2}, we write its general unfolding as
\begin{equation}
\begin{aligned}
		\left( \begin{array}{c}
		\dot{x}\\
		\dot{y}\\
	\end{array} \right) =\left\{
    \begin{aligned}
&		\left( \begin{array}{c}
			f^+(x,y,{\boldsymbol \alpha}^+)\\
			g^+(x,y,{\boldsymbol \alpha}^+)
		\end{array} \right) ~~~~&&\mathrm{if}~(x,y)\in\Sigma^+,\\
&		\left( \begin{array}{c}
			f^-(x,y,{\boldsymbol \alpha}^-)\\
			g^-(x,y,{\boldsymbol \alpha}^-)
		\end{array} \right)   && \mathrm{if}~(x,y)\in\Sigma^-,
	\end{aligned} \right.
\end{aligned}
\label{pws2-unfold}
\end{equation}
where ${\boldsymbol \alpha}^{\pm}\in\mathbb{R}^{n^{\pm}}$ for positive integers $n^{\pm}$. Assume that $f^\pm,~g^\pm$ are $C^\infty$ with respect to $(x,y,{\boldsymbol \alpha}^\pm)$ and system~\eqref{pws2-unfold} with ${\boldsymbol \alpha}^{\pm}={\boldsymbol 0}$ corresponds to system~\eqref{pws2}.

When $m^+=0$, the origin $O$ is a regular point of the upper subsystem of \eqref{pws2}, i.e., $g^+(0,0,{\boldsymbol 0})\ne0$. This implies that the upper subsystem of \eqref{pws2-unfold} has no tangent points in a small neighborhood of $O$.

When $m^+\ge1$, finding tangent points of the upper subsystem is equivalent to
finding zeros of $G^+\left(x,{\boldsymbol \alpha}^+\right)$, where
$G^+\left(x,{\boldsymbol \alpha}^+\right):=g^+(x,0,{\boldsymbol \alpha}^+)$. Clearly,
\begin{equation*}
G^+(x,{\boldsymbol 0})=\beta_{m^+}x^{m^+}+O\left(x^{{m^+}+1}\right), ~~\beta_{m^+}\ne 0.
\end{equation*}
By the Malgrange Preparation Theorem, there exists $\widetilde G^+(x,{\boldsymbol \alpha}^+)$ satisfying $\widetilde G^+(0,{\boldsymbol 0})=\beta_{m^+}$ and $\xi_i({\boldsymbol \alpha}^+)$ satisfying $\xi_i({\boldsymbol 0})=0$ ($i=0,...,m^+-1$) such that
\begin{equation*}
G^+\left(x,{\boldsymbol \alpha}^+\right) = \widetilde G^+\left(x,{\boldsymbol \alpha}^+\right)\left(\xi_0\left({\boldsymbol \alpha}^+\right)+\xi_1\left({\boldsymbol \alpha}^+\right)x+...+\xi_{m^+-1}\left({\boldsymbol \alpha}^+\right)x^{m^+-1}+x^{m^+}\right)
\end{equation*}
on a neighborhood of $\left(x,{\boldsymbol \alpha}^+\right)=(0,{\boldsymbol 0})$.
Thus, $G^+\left(x,{\boldsymbol \alpha}^+\right)$ has at most $m^+$ real roots in a neighborhood,
which means that the upper subsystem has at most
$m^+$ bifurcating tangent points.
Similarly, the lower subsystem has at most
$m^-$ bifurcating tangent points. Therefore, the number $\ell$ of bifurcating tangent points of
system~\eqref{pws2-unfold} is at most $m^++m^-$, i.e., $\ell\in[1,m^++m^-]$.
Then the $\ell$ bifurcating tangent points $(x^*_i,0)$ of multiplicity $(m^+_i, m^-_i)$ ($i=1,...,\ell$)
satisfy (\ref{maxi-sum})
because $x^*_i$ is a zero with multiplicity $m^+_i$ of a polynomial function of degree $m^+$
and a zero with multiplicity $m^-_i$ of a polynomial function of degree $m^-$ by the Malgrange Preparation Theorem.
Here $m^+_i+m^-_i\ge 1$ and one of them may be $0$.
Theorem~\ref{thm1}(a) is proved.

If $\ell=m^++m^-$, by Theorem~\ref{thm1}(a) the $m^++m^-$ bifurcating tangent points $(x^*_i,0)$ ($i=1,...,m^++m^-$) of system~\eqref{pws2-unfold} satisfy
$$\sum_{i=1}^{m^++m^-}\left(m^+_i+m^-_i\right)=m^-+m^+.$$
Associated with each $m^+_i+m^-_i\ge 1$, we get $(m^+_i,m^-_i)\in\left\{(0,1), (1,0)\right\}$ for all $i=1,...,m^++m^-$.
Moreover, $m^+$ elements of $\{x^*_1,...,x^*_{m^++m^-}\}$ are simple zeros of
a polynomial function of degree $m^+$ and the other $m^-$ elements are simple zeros of
a polynomial function of degree $m^-$, which correspond to
$m^+$ tangent points of the upper subsystem and $m^-$ tangent points of the lower one of \eqref{pws2-unfold} respectively.
Clearly, the $m^+$ (resp. $m^-$) simple zeros of that polynomial function of degree $m^+$ (resp. $m^-$) have  alternating monotonicity, which implies the alternating visibility of the $m^+$ (resp. $m^-$) tangent points
because the visibility depends on the sign of derivative of that polynomial function at zero.
Theorem~\ref{thm1}(b) is proved.

To prove Theorem~\ref{thm1}(c), we consider unfolding \eqref{pws2-unfold} of system~\eqref{pws2} in a special form
\begin{equation}
\begin{aligned}
        \left( \begin{array}{c}
        \dot{x}\\
        \dot{y}\\
    \end{array} \right) =\left\{
    \begin{aligned}
&       \left( \begin{array}{c}
            f^+(x,y)\\
            \phi^+(x,y)\prod_{i=1}^{m^+}(x-{\boldsymbol \alpha}^+_i)+\Upsilon^+(x,y)y
        \end{array} \right) ~~~~&&\mathrm{if}~(x,y)\in\Sigma^+,\\
&       \left( \begin{array}{c}
            f^-(x,y)\\
            \phi^-(x,y)\prod_{i=1}^{m^-}(x-{\boldsymbol \alpha}^-_i)+\Upsilon^-(x,y)y
        \end{array} \right)   && \mathrm{if}~(x,y)\in\Sigma^-,
    \end{aligned} \right.
    \end{aligned}
    \label{pws2-exunfold}
\end{equation}
where ${\boldsymbol \alpha}^{\pm}:=\left({\boldsymbol \alpha}^{\pm}_i\right)\in\mathbb{R}^{m^{\pm}}$.
For given $\ell\in [1,m^++m^-]$, one can easily choose appropriate ${\boldsymbol \alpha}^\pm$ such that
system~\eqref{pws2-exunfold} has exactly $\ell$ bifurcating tangent points and two ``=" in (\ref{maxi-sum}) hold.
For example, we can take ${\boldsymbol \alpha}^{\pm}$ such that
$$
\prod_{i=1}^{m^+}(x-{\boldsymbol \alpha}^+_i)=\prod_{i=1}^{\ell^+}(x-\widehat{\boldsymbol \alpha}^+_i)^{\widehat m^+_i},~~
\prod_{i=1}^{m^-}(x-{\boldsymbol \alpha}^-_i)=\prod_{i=1}^{\ell^-}(x-\widehat{\boldsymbol \alpha}^-_i)^{\widehat m^-_i},
$$
where $\ell^\pm\le m^\pm$, $\ell^++\ell^-=\ell$ and $\ell^+=0$ (resp. $\ell^-=0$) if $m^+=0$ (resp. $m^-=0$),
all $\widehat{\boldsymbol \alpha}^+_i, \widehat{\boldsymbol \alpha}^-_i$ are distinct from each other.
Then \eqref{pws2-exunfold} has exactly $\ell$ bifurcating tangent points
$(\widehat{\boldsymbol \alpha}^+_i, 0)$ of multiplicity $(\widehat m^+_i, 0)$ and
$(\widehat{\boldsymbol \alpha}^-_i, 0)$ of multiplicity $(0, \widehat m^-_i)$.
Two ``=" in (\ref{maxi-sum}) hold because
$\widehat m^+_1+...+\widehat m^+_{\ell^+}=m^+, \widehat m^-_1+...+\widehat m^-_{\ell^-}=m^-$.
Theorem~\ref{thm1}(c) is proved.
\end{proof}

\begin{proof}[Proof of Theorem~\ref{thm2}]
Since $f^+(0,0)\ne 0$, we assume that $f^+(0,0)>0$ without loss of generality.
In the following we prove this theorem for different cases of the set of $\ell$.

{\it Step 1. Reachability for each $\ell\in\left\{1\right\}$}.
This case corresponds to four subcases: $O$ of system~\eqref{pws3} is either invisible with $m^+=3$ or visible with $m^+=1$ or left with $m^+=2$ or right with $m^+=2$. Clearly, Theorem~\ref{thm2} holds because $\ell=1$.

{\it Step 2. Reachability for each $\ell\in\left\{1,2\right\}$}.
This case corresponds to four subcases:  $O$ of system~\eqref{pws3} is either invisible with $m^+=5$ or visible with $m^+=3$ or left with $m^+=4$ or right with $m^+=4$.
For the first subcase, we take $0<\lambda^+_1<\lambda^+_2<\lambda^+_4<\lambda^+_5$ and $\lambda^+_3\in [\lambda^+_2, \lambda^+_4]$
in transition system~\eqref{tpws}.
Since $O$ is invisible, we get that besides invisible bifurcating
tangent points $\left(\lambda^+_1,0\right), \left(\lambda^+_5,0\right)$
there are bifurcating tangent points $\left(\lambda^+_2,0\right), \left(\lambda^+_3,0\right)$ and $\left(\lambda^+_4,0\right)$.
Here $\left(\lambda^+_3,0\right)$ maybe coincide with $\left(\lambda^+_2,0\right)$ or $\left(\lambda^+_4,0\right)$.
Let function $y=\Theta(x)$ be the solution of the  Cauchy problem
\begin{equation*}
\frac{dy}{dx}=\frac{\widehat Z^+_2(x,y)}{\widehat Z^+_1(x,y)},~~~~
y\left(\lambda^+_2\right)=0.
\end{equation*}
Then, $\Theta\left(\lambda^+_4\right)<0$ when $\lambda^+_3=\lambda^+_2$ and, hence, $\Theta\left(\lambda^+_4\right)<0$ when $0<\lambda^+_3-\lambda^+_2\ll 1$ by
the continuity of $\Theta(x)$ with respect to $\lambda^+_3$.
Thus, orbit $\widehat\gamma^+\left(t,\lambda^+_2,0\right)$ (resp. $\widehat\gamma^+\left(t,\lambda^+_4,0\right)$) only passes through visible bifurcating tangent point $\left(\lambda^+_2,0\right)$ (resp. $\left(\lambda^+_4,0\right)$) as shown in Figure~\ref{Fig-thm2-1}(a).
Therefore, system~\eqref{pws4} with $\psi^+\left(x,{\boldsymbol 0}\right)\equiv 0$ satisfies the reachability for $\ell=1$.

\begin{figure}[htp]
\centering
\subfigure[two bifurcating tangent orbits]
 {
  \scalebox{0.45}[0.45]{
   \includegraphics{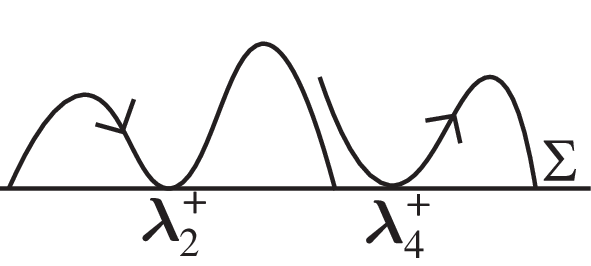}}}~~~~
\subfigure[one bifurcating tangent orbit]
 {
  \scalebox{0.45}[0.45]{
   \includegraphics{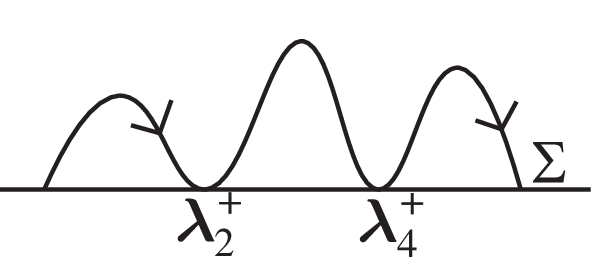}}}
\caption{bifurcating tangent orbit passing through $\left(\lambda^+_2,0\right)$}
\label{Fig-thm2-1}
\end{figure}
On the other hand, we similarly get that
$\Theta\left(\lambda^+_4\right)>0$ when $\lambda^+_3=\lambda^+_4$ and, hence, $\Theta\left(\lambda^+_4\right)>0$
when $0<\lambda^+_4-\lambda^+_3\ll 1$ by
the continuity of $\Theta(x)$ with respect to $\lambda^+_3$.
Thus, there exists $\lambda^+_3\in (\lambda^+_2,\lambda^+_4)$ such that $\Theta\left(\lambda^+_4\right)=0$ by the Intermediate Value Theorem.
This implies that orbit $\widehat\gamma^+\left(t,\lambda^+_2,0\right)$ passes through visible bifurcating tangent point $\left(\lambda^+_4,0\right)$
as shown in Figure~\ref{Fig-thm2-1}(b). Therefore, system~\eqref{pws4}
with $\psi^+\left(x,{\boldsymbol 0}\right)\equiv 0$ satisfies the reachability for $\ell=2$.
The reachability can be proved similarly for the other three subcases and we omit the statements here.

{\it Step 3. Reachability for each $\ell\in\{1,...,r\}$ with $r\ge3$.}
This case corresponds to four subcases: $O$ of \eqref{pws2} is either invisible with $m^+=2r+1$ or visible with $m^+=2r-1$ or left with $m^+=2r$ or right with $m^+=2r$. For the first subcase,
we take $\lambda^+_i=i\delta$ for small $\delta>0$ ($i=1,..,m^+$) in transition system~\eqref{tpws} and, hence,
$\left(\lambda^+_{2i},0\right)$ ($i=1,...,(m^+-1)/2$) are visible bifurcating tangent points. Consider functions $y=\Theta_n(x)$
($n=1,...,(m^+-1)/2$) satisfying the Cauchy problems respectively
\begin{equation*}
\frac{dy}{dx}=\frac{\widehat Z^+_2(x,y)}{\widehat Z^+_1(x,y)},~~~~
y\left(\lambda^+_2\right)=C_n,
\end{equation*}
where $C_n:=\delta^{m^+}(1+n\delta)$.
We obtain that for all $x\in\left(\lambda^+_2,\lambda^+_{m^+}\right)=(2\delta, m^+\delta)$
\begin{equation}
\begin{aligned}
\Theta_n(x)
& = \Theta_n\left(\lambda^+_2\right)+\int_{\lambda^+_2}^{x}\frac{\widehat Z^+_2(s,\Theta_n(s))}{\widehat Z^+_1(s,\Theta_n(s))}ds\\
& = \delta^{m^+}(1+n\delta)+\left(x-\lambda^+_2\right)\frac{\widehat Z^+_2(\sigma_n,\Theta_n(\sigma_n))}{\widehat Z^+_1(\sigma_n,\Theta_n(\sigma_n))}\\
& = \delta^{m^+}(1+n\delta)+\left(x-\lambda^+_2\right)\frac{\phi^+\left(\sigma_n,\Theta_n(\sigma_n)\right)}{\widehat Z^+_1\left(\sigma_n,\Theta_n(\sigma_n)\right)}\prod_{i=1}^{m^+}\left(\sigma_n-\lambda^+_i\right)
\\
& = \delta^{m^+}(1+n\delta)+\delta^{m^++1}(x^*-2)\frac{\phi^+\left(\sigma^*_n\delta,\Theta_n(\sigma^*_n\delta)\right)}{\widehat Z^+_1\left(\sigma^*_n\delta,\Theta_n(\sigma^*_n\delta)\right)}\prod_{i=1}^{m^+}(\sigma^*_n-i)\\
& = \delta^{m^+}+O\left(\delta^{m^++1}\right)\\
&>0
\end{aligned}
\label{Theta-Esti}
\end{equation}
for sufficiently small $\delta>0$, where $x^*=x/\delta\in\left(2,m^+\right)$, $\sigma_n\in\left(\lambda^+_2,x\right)$ and $\sigma^*_n=\sigma_n/\delta\in\left(2,x^*\right)$.
Thus, each orbit $\widehat\gamma^+\left(t,\lambda^+_2,C_n\right)$ ($n=1,...,(m^+-1)/2$)
horizontally intersects with lines $x=\lambda^+_m$  ($m=1,...,m^+$) in $\Sigma^+$
at some $t$ (denoted as $t_{n,m}$),
i.e., $Y_{n,m}:=\widehat\gamma^+_2\left(t_{n,m},\lambda^+_2, C_n \right)>0$.
Clearly, $t_{n,2}=0$ and $Y_{n,2}=C_n$ for all $n=1,...,(m^+-1)/2$.

For the reachability for $\ell=1$, function $\psi^+\left(x,{\boldsymbol k}^+\right)$ is defined by taking $d=(m^+-1)/2$ and
\begin{equation}
\left\{
\begin{aligned}
&k^+_i=\lambda^+_i,&&~{\rm~for~}i=1,...,2d+1,\\
&k^+_{n+2d+1}=Y_{n,2n},&&~{\rm~for~}n=1,...,d.
\end{aligned}
\right.
\label{kYrelation}
\end{equation}
Clearly, ${\boldsymbol k}^+\in{\cal K}$ and it satisfies~\eqref{K} because
\begin{equation*}
Y_{n,2n}=\widehat\gamma^+_2\left(t_{n,2n},\lambda^+_2,C_n\right)=\delta^{m^+}+O\left(\delta^{m^++1}\right).
\end{equation*}
Meanwhile, it is not hard to check that
\begin{equation*}
\widetilde Z^+_1\left(\lambda^+_i,0\right) = \widehat Z^+_1\left(\lambda^+_i,\psi^+\left(\lambda^+_i,{\boldsymbol k}^+\right)\right)>0
\end{equation*}
by continuity and
\begin{equation*}
\begin{aligned}
\widetilde Z^+_2\left(\lambda^+_i,0\right)
& = \widehat Z^+_2\left(\lambda^+_i,\psi^+\left(\lambda^+_i,{\boldsymbol k}^+\right)\right)-\widehat Z^+_1\left(\lambda^+_i,\psi^+\left(\lambda^+_i,{\boldsymbol k}^+\right)\right)\dot\psi^+\left(\lambda^+_i,{\boldsymbol k}^+\right)\\
& = \widehat Z^+_2\left(\lambda^+_i,\psi^+\left(\lambda^+_i,{\boldsymbol k}^+\right)\right)\\
& = 0
\end{aligned}
\end{equation*}
by $\dot\psi^+\left(\lambda^+_i,{\boldsymbol k}^+\right)=0$, where $i=1,...,m^+$. Further, by $\ddot\psi^+\left(\lambda^+_i,{\boldsymbol k}^+\right)=0$ we obtain
\begin{equation*}
{\rm sgn}\left(\widetilde Z^+_1\left(\lambda^+_i,0\right)\frac{\partial \widetilde Z^+_2}{\partial x}\left(\lambda^+_i,0\right)\right)={\rm sgn}\left(\widehat Z^+_1\left(\lambda^+_i,0\right)\frac{\partial \widehat Z^+_2}{\partial x}\left(\lambda^+_i,0\right)\right),
\end{equation*}
i.e., bifurcating tangent points of system~\eqref{pws4} are same with ones of transition system~\eqref{tpws} including locations and types.

For orbit $\widehat\gamma^+\left(t,\lambda^+_2,Y_{1,2}\right)$, there is a corresponding orbit $\widetilde\gamma^+\left(t,\lambda^+_2,Y_{1,2}-\psi^+\left(\lambda^+_2,{\boldsymbol k}^+\right)\right)$
by Proposition~\ref{prop2}.
Moreover, $Y_{1,2}-\psi^+\left(\lambda^+_2,{\boldsymbol k}^+\right)=0$ because of (\ref{psidef}) and (\ref{kYrelation}).
Then $\widetilde\gamma^+_1\left(t,\lambda^+_2,0\right) = \widehat\gamma^+_1\left(t,\lambda^+_2,Y_{1,2}\right)$. Meanwhile,
\begin{equation*}
\begin{aligned}
\widetilde\gamma^+_2\left(t,\lambda^+_2,0\right)
& = \widehat\gamma^+_2\left(t,\lambda^+_2,Y_{1,2}\right)-\psi^+\left(\widehat\gamma^+_1\left(t,\lambda^+_2,Y_{1,2}\right),{\boldsymbol k}^+\right)\\
& \ge \widehat\gamma^+_2\left(t,\lambda^+_2,Y_{1,2}\right)-\psi^+\left(\lambda^+_2,{\boldsymbol k}^+\right)\\
& = \widehat\gamma^+_2\left(t,\lambda^+_2,Y_{1,2}\right)-Y_{1,2}\\
& \ge 0.
\end{aligned}
\end{equation*}
for $t\in[t_{1,1},t_{1,3}]$ and $\widetilde\gamma^+_2\left(t,\lambda^+_2,0\right)=0$ if and only if $t=t_{1,2}$. On the other hand, since
\begin{equation*}
\begin{aligned}
\widetilde\gamma^+_2\left(t_{1,4},\lambda^+_2,0\right)
& = \widehat\gamma^+_2\left(t_{1,4},\lambda^+_2,Y_{1,2}\right)-\psi^+\left(\widehat\gamma^+_1\left(t_{1,4},\lambda^+_2,Y_{1,2}\right),{\boldsymbol k}^+\right)\\
& = \widehat\gamma^+_2\left(t_{1,4},\lambda^+_2,Y_{1,2}\right)-\psi^+\left(\lambda^+_4,{\boldsymbol k}^+\right)\\
& = \widehat\gamma^+_2\left(t_{1,4},\lambda^+_2,Y_{1,2}\right)-\widehat\gamma^+_2\left(t_{2,4},\lambda^+_2,Y_{2,2}\right)\\
& < 0,
\end{aligned}
\end{equation*}
orbit $\widetilde\gamma^+\left(t,\lambda^+_2,0\right)$ transversal intersects with $\Sigma$ at $(Q,0)$ at some $t_{1,3}<t<t_{1,4}$, i.e., $\lambda^+_3<Q<\lambda^+_4$. Thus, it only passes through one visible bifurcating tangent point $\left(\lambda^+_2,0\right)$.

It is proved similarly that orbits $\widetilde\gamma^+\left(t,\lambda^+_2,Y_{n,2}-Y_{1,2}\right)$ pass through one visible bifurcating tangent point $\left(\lambda^+_{2n},0\right)$ ($n=2,...,(m^+-1)/2$) respectively. Thus, for system~\eqref{pws4} there are $(m^+-1)/2$ bifurcating tangent
orbits in $\Sigma^+$ and each of them passes through one visible bifurcating tangent point, which proves the reachability for $\ell=1$.
An example for $m^+=7$ is shown in Figure~\ref{Fig-thm2-2}, where the dashed curve in Figure~\ref{Fig-thm2-2}(a)
denotes $\psi^+\left(x,{\boldsymbol k}^+\right)$.
\begin{figure}[htp]
\centering
\subfigure[transition system~\eqref{tpws}]
 {
  \scalebox{0.35}[0.35]{
   \includegraphics{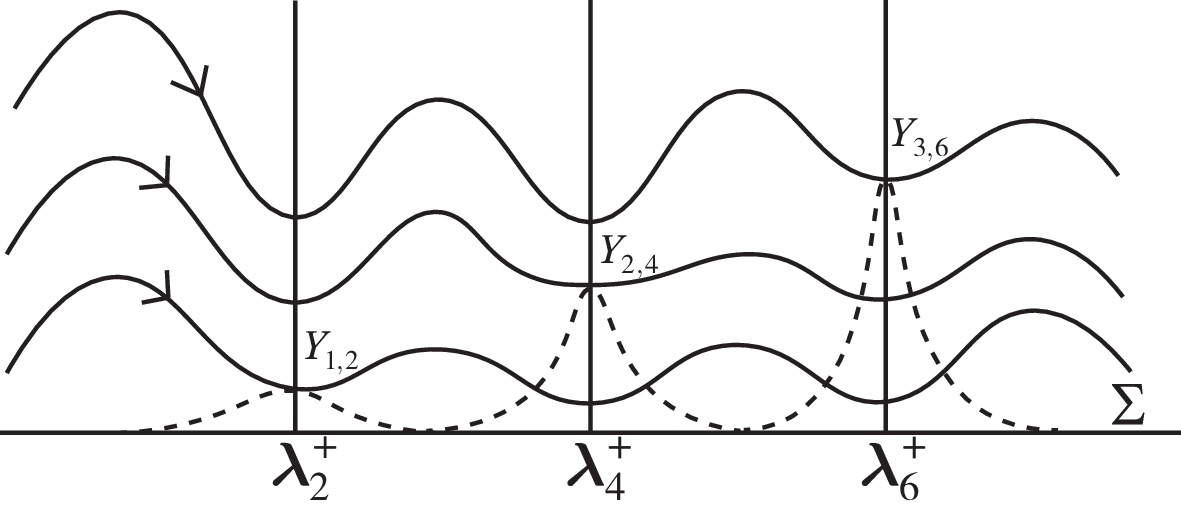}}}
\subfigure[system~\eqref{pws4}]
 {
  \scalebox{0.35}[0.35]{
   \includegraphics{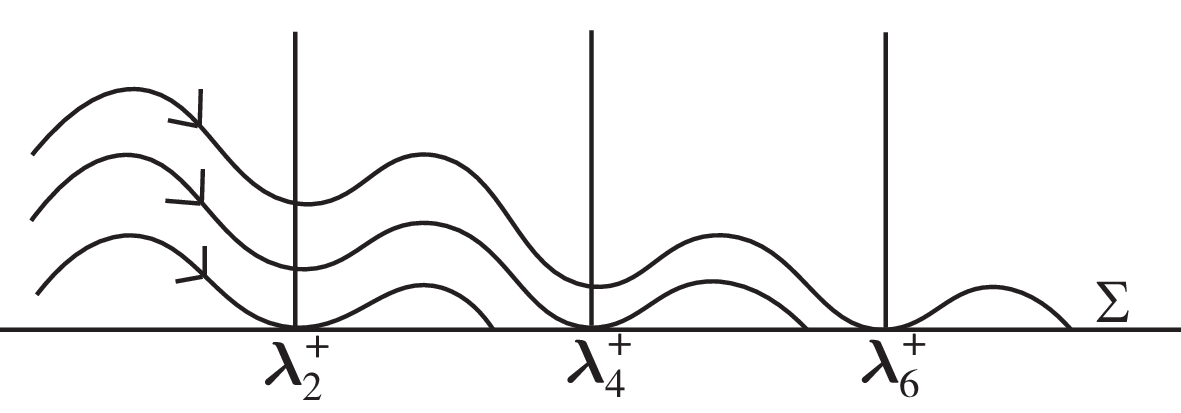}}}
\caption{the example for $m^+=7$}
\label{Fig-thm2-2}
\end{figure}

For the reachability for $\ell\ge2$, we take
$d=(m^+-1)/2$,
$k^+_i=\lambda^+_i$ ($i=1,...,2d+1$) and for $n=1,...,d$
\begin{equation*}
k^+_{n+2d+1}=
\left\{
\begin{aligned}
&Y_{j,2n},&&{\rm~for~}n=1,...,\lfloor d/\ell\rfloor\ell,\\
&2\delta^{m^+},&&{\rm~for~}n>\lfloor d/\ell\rfloor\ell,
\end{aligned}
\right.
\end{equation*}
where $j=\lfloor (n-1)/\ell\rfloor+1$.
By a similar way, we obtain that orbit $\widetilde\gamma^+\left(t,\lambda^+_2,Y_{j,2}-Y_{1,2}\right)$ ($j=1,...,\lfloor d/\ell\rfloor$)
passes through visible bifurcating tangent points $\left(\lambda^+_{2m},0\right)$ ($m=1+(j-1)\ell,...,j\ell$)
and transversal intersects with $\Sigma$ at $\left(Q_j,0\right)$ respectively and $\lambda^+_{2j\ell+1}<Q_j<\lambda^+_{2j\ell+2}$. Thus, the reachability for any $\ell\ge 2$ is proved.

For the other three subcases, analysis is similar and we omit the statements.
\end{proof}

%%%%%%%%%%%%%%%%%%%%%%%%%%%%%%%%%%%%%%%%%%%%%%%%%%%%%%%%%%%
\section{Proofs of Theorems~\ref{thm3}, \ref{thm4} and \ref{thm5}}
\setcounter{equation}{0}
\setcounter{lm}{0}
\setcounter{thm}{0}
\setcounter{rmk}{0}
\setcounter{df}{0}
\setcounter{cor}{0}

In this section we prove the results on bifurcating loops of $L^{cri}_*$ including $L_c,~L_s,~L^{cro},~L^{cri}$, i.e., Theorems~\ref{thm3}, \ref{thm4} and \ref{thm5}, where $L^{cri}_*$ is given in the beginning of section 3.2.
We begin with transition map for  system
\begin{equation}
        \left( \begin{array}{c}
        \dot{x}\\
        \dot{y}\\
    \end{array} \right)
    =
      \left( \begin{array}{c}
            f(x,y)\\
            g(x,y)
        \end{array} \right)=:{\cal Z}(x,y),
\label{trans-sys}
\end{equation}
where $(x,y)\in{\mathbb R}^2$ and $f, g$ are $C^{\infty}$. Let
\begin{equation*}
\gamma(t,x_0,y_0):=\left(\gamma_1(t,x_0,y_0),\gamma_2(t,x_0,y_0)\right)^\top
\end{equation*}
be the solution of system~\eqref{trans-sys} with initial value $(x_0,y_0)$. For regular point $(x_0,y_0)$ and a given $T>0$, denote $\gamma(T,x_0,y_0)$ by $(x_1,y_1)^\top$.
We take straight crossing segment $S_0$ (resp. $S_1$) at $(x_0,y_0)$ (resp. $(x_1,y_1)$) and
let $N_0:=(N_{01},N_{02})^\top$ (resp. $N_1:=(N_{11},N_{12})^\top$) be the unit vector parallel to $S_0$ (resp. $S_1$).
Point $\left(\widetilde x_0,\widetilde y_0\right)$ sufficiently close to $(x_0,y_0)$ on
$S_0$ is written as $\left(x_0+rN_{01}, y_0+rN_{02}\right)$
and then orbit $\gamma\left(t,\widetilde x_0,\widetilde y_0\right)$ intersects $S_1$ at
\begin{equation*}
\left(\widetilde x_1,\widetilde y_1\right)=\left(x_1+V(r;x_0,y_0,T,S_0,S_1)N_{11}, y_1+V(r;x_0,y_0,T,S_0,S_1)N_{12}\right)
\end{equation*}
for $t=T(r)$ satisfying $T(0)=T$ as shown in Figure~\ref{Fig-trans}, where $r$ is sufficiently small.
\begin{figure}[htp]
\centering
   \includegraphics[width=0.52\textwidth]{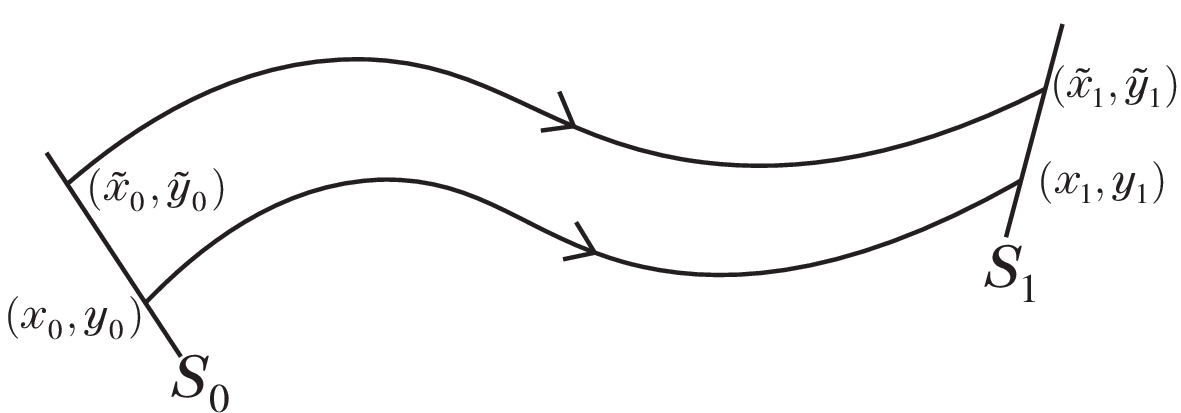}
\caption{classic transition map}
\label{Fig-trans}
\end{figure}
Function $V(r;x_0,y_0,T,S_0,S_1)$ is called a {\it transition map} as in \cite{ZhangJ98}
and is generally used to analyze dynamical behaviors of limit cycles and homoclinic orbits
of system~\eqref{trans-sys}. As in \cite{JKHaleBook,KuznetsovBook,ZhangJ98}, for system~\eqref{trans-sys}
transition map $V(r;x_0,y_0,T,S_0,S_1)$ is $C^\infty$ in $r$ and
\begin{equation}
V(r;x_0,y_0,T,S_0,S_1)=V_1r+O(r^2)
\label{V-Expan}
\end{equation}
for $0<r\ll1$, where
\begin{equation*}
V_1:=\frac{\Delta_0}{\Delta_1}\exp\left\{\int_0^{T}\frac{\partial f\left(\gamma\left(s,x_0,y_0\right)\right)}{\partial x}+\frac{\partial g\left(\gamma\left(s,x_0,y_0\right)\right)}{\partial y}ds\right\}\ne0
\end{equation*}
and
\begin{eqnarray}
\Delta_0:=\det\left({\cal Z}(x_0,y_0),N_0\right), ~~~~\Delta_1:=\det\left({\cal Z}(x_0,y_0),N_1\right).
\label{D0D1}
\end{eqnarray}

\begin{proof}[Proof of Theorem~\ref{thm3}]
Without loss of generality, we assume that $m^+\ge m^-$.
As mentioned below (H), there are 14 types of $L^{cri}_*$ as shown in Figure~\ref{Fig-Loop}. For the type of Figure~\ref{Fig-Loop}(a),
we take ${\boldsymbol \lambda}^-={\boldsymbol 0}$ and let
\begin{equation*}
S_0:=\left\{\left(C,y\right):~y\in(-\epsilon,\epsilon)\right\},~~S_1:=\left\{\left(x,0\right):~x\in\left(P_x-\epsilon,P_x+\epsilon\right)\right\},
\end{equation*}
where $C$ is undetermined and $P^*:(P_x,0)$ is the intersection point of $L^{cri}_*$ on $\Sigma$ besides $O$.

To prove conclusion (a), we first consider the perturbations of the upper subsystem.
For the case that $m^+=3$, we take $\psi^+\left(x,{\boldsymbol k}^+\right)\equiv 0$ and
$\lambda^+_1<\lambda^+_2<\lambda^+_3<0$, which implies that $\left(\lambda^+_1,0\right)$ and $\left(\lambda^+_3,0\right)$ are visible bifurcating tangent points of transition system~\eqref{tpws}. By the proof of Theorem~\ref{thm2}, there exists $\lambda^+_2\in(\lambda^+_1,\lambda^+_3)$
such that orbit $\widetilde\gamma^+\left(t,\lambda^+_1,0\right)$ transversally intersects $\Sigma$ at $\left(Q_a,0\right)$ for forward direction and at $\left(P^+_a,0\right)$ for backward direction as shown in Figure~\ref{Fig-thm3}(a). Moreover, $\lambda^+_2<Q_a<\lambda^+_3<0$ and $\left(P^+_a,0\right)\in S_1$ as $|{\boldsymbol \lambda}^+|$ is sufficiently small.

For the case that $m^+\ge 5$, we take $\lambda^+_i=\left(i-m^+-1\right)\delta$ ($i=1,...,m^+$) for small $\delta>0$, which implies that $\left(\lambda^+_{2i-1},0\right)$ ($i=1,...,(m^++1)/2$) are visible bifurcating tangent points for transition system~\eqref{tpws}. For any $\ell\in\{1,...,\lfloor m^+/2\rfloor\}$, there exists $\psi^+\left(x,{\boldsymbol k}^+\right)$ such that there is an orbit $\widetilde\gamma^+\left(t,\lambda^+_1,0\right)$ passing through visible bifurcating tangent points $\left(\lambda^+_{2i-1},0\right)$ ($i=1,...,\ell$) by the proof of Theorem~\ref{thm2}. Further, it transversally intersects $\Sigma$ at $(P^+_a,0)$ for backward direction and at $(Q_a,0)$ for forward direction. Moreover, $\lambda^+_{2\ell}<Q_a<\lambda^+_{2\ell+1}<0$ and $(P^+_a,0)\in S_1$ as $\delta$ is sufficiently small.

Then we consider the lower subsystem of \eqref{tpws}, orbit $\widehat\gamma^-\left(t,Q_a,0\right)$ transversally intersects $\Sigma$ at $(P^-_a,0)$ for forward direction and $(P^-_a,0)\in S_1$ as $|{\boldsymbol \lambda}^+|$ is sufficiently small because $Q_a\to 0$ as $|{\boldsymbol \lambda}^+|\to 0$. Further, we take $C=Q_a$ in sector $S_0$ and define
\begin{equation}
D(y;Q_a):=P^-_a+V^-(y;Q_a,0,T^-_a,S_0,S_1)-P^+_a,
\label{Dydef}
\end{equation}
where $T^-_a$ satisfies that $\widehat\gamma^-\left(T^-_a,Q_a,0\right)=\left(P^-_a,0\right)^\top$. By \eqref{V-Expan}, $D(y;Q_a)=P^-_a-P^+_a+V^-_1y+O\left(y^2\right)$
implying that for any sufficiently small ${\boldsymbol \lambda}^+$ there is one $y_0\in(-\epsilon,\epsilon)$ such that $D(y_0;Q_a)=0$.
That is, there is a $T^-_a=T^-_a(y_0)>0$ such that
$\widehat\gamma^-\left(T^-_a(y_0),Q_a,y_0\right)=\left(P^+_a,0\right)^\top$
as shown in Figure~\ref{Fig-thm3}(a) for $y_0<0$.
Note that $L^{cro}(1)$ is already obtained if $y_0=0$.
\begin{figure}[htp]
\centering
\subfigure[$\widehat\gamma^-\left(t,Q_a,0\right)$ and $\widehat\gamma^-\left(t,Q_a,y_0\right)$]
 {
  \scalebox{0.4}[0.4]{
   \includegraphics{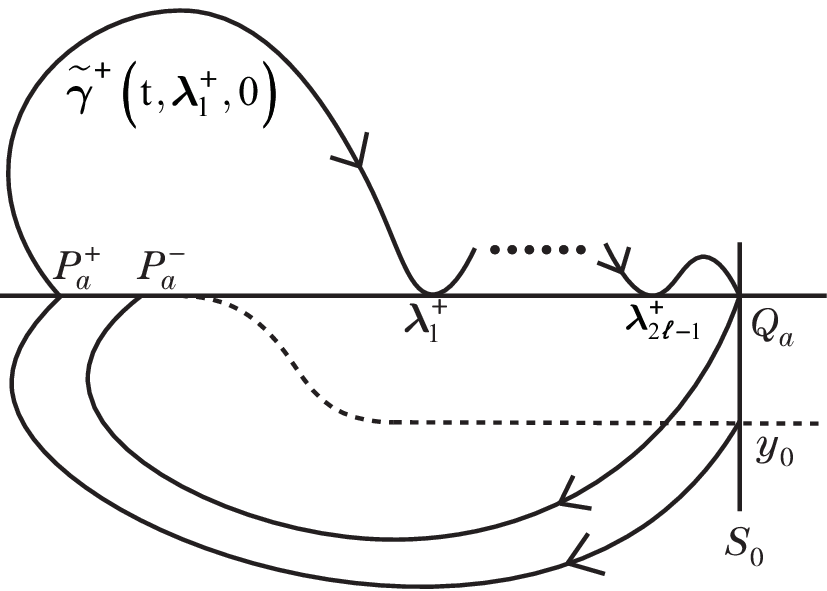}}}~~~~
\subfigure[$\widetilde\gamma^-\left(t,Q_a,0\right)$]
 {
  \scalebox{0.4}[0.4]{
   \includegraphics{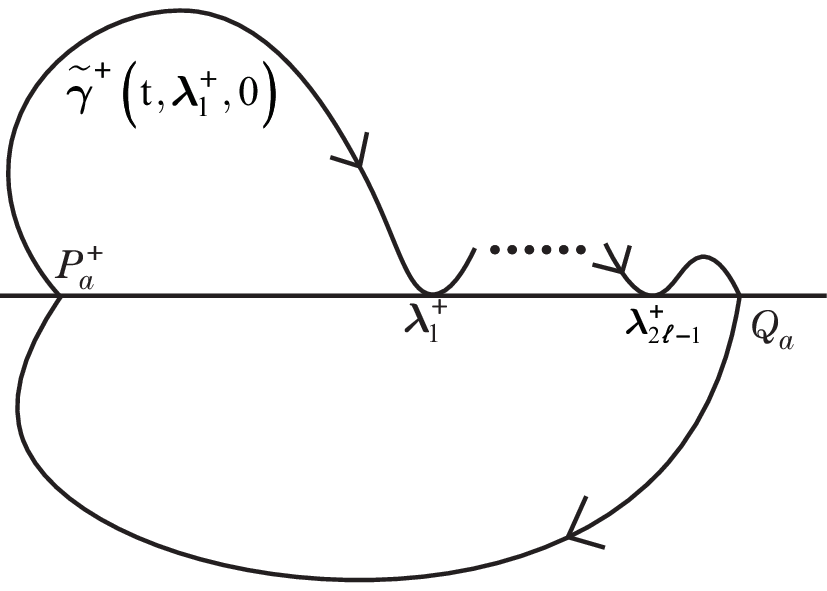}}}
\caption{$L^{cro}(\ell)$ for $y_0<0$}
\label{Fig-thm3}
\end{figure}
We define function $\psi^-\left(x,{\boldsymbol k}^-\right)$ by taking
\begin{eqnarray}
d=1,~~k^-_1=k^-_2=k^-_3=0,~~ k^-_4=y_0,~~
r_1=\frac{2P_x}{3},~~ r_2=\frac{P_x}{3}
\label{dkr12}
\end{eqnarray}
as shown by the dashed curve in Figure~\ref{Fig-thm3}(a).
Clearly, ${\boldsymbol k}^-\in{\mathbb R}^{4}\setminus{\cal K}$ and bifurcating tangent points of system~\eqref{pws4} are same with ones of transition system~\eqref{tpws}.
By $Q_a\in (\lambda_2^+, \lambda_3^+)$ for $m^+=3$ (resp.
$Q\in (\lambda_{2\ell}^+, \lambda_{2\ell+1}^+)$ for $m^+\ge 3$) and the definition of $\psi^-\left(x,{\boldsymbol k}^-\right)$, we get
$\psi^-\left(Q_a,{\boldsymbol k}^-\right)=k_4^-=y_0$.
Associated with Proposition~\ref{prop2}, corresponding orbit $\widetilde\gamma^-\left(t,Q_a,y_0-\psi^-\left(Q_a,{\boldsymbol k}^-\right)\right)$
of orbit $\widehat\gamma^-\left(t,Q_a,y_0\right)$ is exactly $\widetilde\gamma^-\left(t,Q_a,0\right)$.
Moreover, $\widetilde\gamma^-_1\left(t,Q_a,0\right)=\widehat\gamma^-_1\left(t,Q_a,y_0\right)$,
$\widetilde\gamma^-_2\left(t,Q_a,0\right) = \widehat\gamma^-_2\left(t,Q_a,y_0\right)-\psi^-\left(\widehat\gamma^+_1\left(t,Q_a,y_0\right),{\boldsymbol k}^-\right)< 0$
over interval $\left(0,T^-_a(y_0)\right)$ and $\widetilde\gamma^-_2\left(0,Q_a,0\right)=\widetilde\gamma^-_2\left(T^-_a(y_0),Q_a,0\right)=0$.
That is, orbit $\widetilde\gamma^-\left(t,Q_a,0\right)$ transversally intersects $\Sigma$ at point $\left(P^+_a,0\right)$ for forward direction as shown in Figure~\ref{Fig-thm3}(b).
Therefore, system~\eqref{pws4} has a $L^{cro}$ connecting visible bifurcating tangent point $\left(\lambda_1^+,0\right)$ for $m^+=3$
(resp. visible bifurcating tangent points $\left(\lambda^+_{2i-1},0\right)$ with $i=1,...,\ell$ for $m^+\ge 5$)
and transversally intersecting $\Sigma$ at $\left(Q_a,0\right), \left(P^+_a,0\right)$. The proof is the same if $y_0>0$.
The existence of $L^{cro}(\ell)$ is proved. The proof of conclusion (a) for the type of Figure~\ref{Fig-Loop}(a) is finished.

In the following we prove conclusion (b) and begin with perturbations of the upper subsystem.
For the case that $m^+=1$, we take $\psi^+\left(x,{\boldsymbol k}^+\right)\equiv0$ and sufficiently small $\lambda^+_1<0$.
Then orbit $\widetilde\gamma^+\left(t,\lambda^+_1,0\right)$ transversally intersects $\Sigma$ at $\left(P^+_b,0\right)\in S_1$ for backward direction as shown in Figure~\ref{Fig-thm3-1}(a).
For the case that $m^+=3$, we take $\psi^+\left(x,{\boldsymbol k}^+\right)\equiv 0$ and sufficiently small
$\lambda^+_1<\lambda^+_2<\lambda^+_3<0$. By the proof of Theorem~\ref{thm2}, for any $\ell\in\{1,2\}$ there exists $\lambda^+_2\in(\lambda^+_1,\lambda^+_3)$ such that
orbit $\widetilde\gamma^+\left(t,\lambda^+_1,0\right)$ passes through $\left(\lambda^+_{2i-1},0\right)$ ($i=1,...,\ell$) and transversally intersects $\Sigma$ at $(P^+_b,0)\in S_1$ for backward direction.
For the case that $m^+\ge 5$, we take $\lambda^+_i=(i-m^+-1)\delta$ ($i=1,...,m^+$) for small $\delta>0$, which implies that $\left(\lambda^+_{2i-1},0\right)$ ($i=1,...,(m^++1)/2$) are visible bifurcating tangent points. For any $\ell\in\{1,...,(m^++1)/2\}$, there exists $\psi^+\left(x,{\boldsymbol k}^+\right)$ such that orbit $\widetilde\gamma^+\left(t,\lambda^+_1,0\right)$ connecting visible bifurcating tangent points $\left(\lambda^+_{2i-1},0\right)$ ($i=1,...,\ell$) by the proof of Theorem~\ref{thm2}. Moreover, it transversally intersects $\Sigma$ at $(P^+_b,0)\in S_1$ for backward direction as $\delta$ is sufficiently small.

Now consider the lower subsystem of \eqref{tpws}. Orbit $\widehat\gamma^-\left(t,\lambda^+_1,0\right)$ transversally intersects $\Sigma$ at $\left(P^-_b,0\right)\in S_1$ for forward direction as $|{\boldsymbol \lambda}^+|$ is sufficiently small.
Taking $C=\lambda^+_{2\ell-1}$ in sector $S_0$ and define
\begin{equation}
D(y;\lambda^+_{2\ell-1}):=P^-_b+V^-(y;\lambda^+_{2\ell-1},0,T^-_b,S_0,S_1)-P^+_b,
\label{Dydef2}
\end{equation}
where $T^-_b$ satisfies that $\widehat\gamma^-\left(T^-_b,\lambda^+_{2\ell-1},0\right)=(P^-_b,0)^\top$.
As the discussion for \eqref{Dydef}, we get one zero $y_0\in(-\epsilon,\epsilon)$ of $D(y;\lambda^+_{2\ell-1})$ and
use it to construct perturbation function $\psi^-\left(x,{\boldsymbol k}^-\right)$ via \eqref{dkr12}.
Then, system~\eqref{pws4} has a $L^{cri}(\ell)$ connecting $\ell$ visible bifurcating tangent points $(\lambda^+_{2i-1},0)$ ($i=1,...,\ell$) as shown in Figure~\ref{Fig-thm3-1}(b).
The proof is same if $y_0>0$ and conclusion(b) for the type in Figure~\ref{Fig-Loop}(a) is proved.

\begin{figure}[htp]
\centering
\subfigure[$\widehat\gamma^-\left(t,\lambda^+_1,0\right)$ and $\widehat\gamma^-\left(t,\lambda^+_1,y_0\right)$]
 {
  \scalebox{0.4}[0.4]{
   \includegraphics{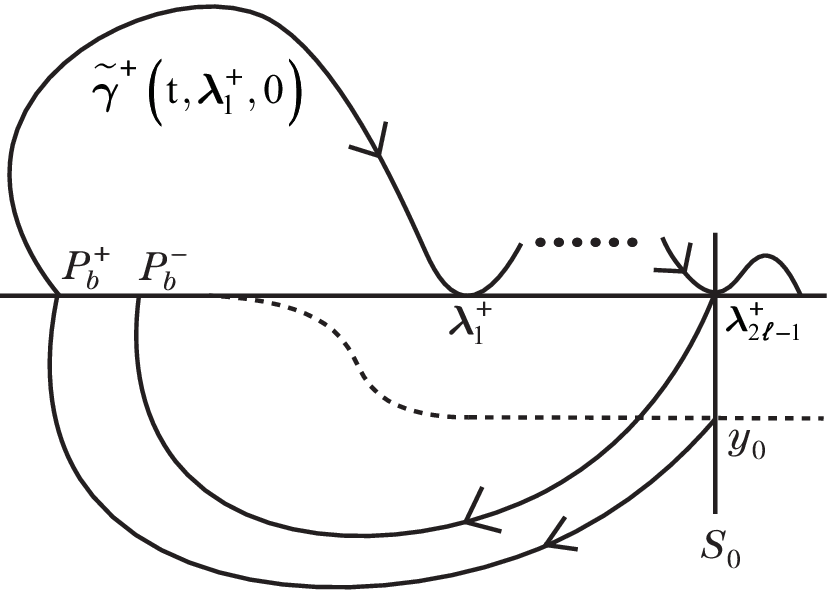}}}~~~~
\subfigure[$\widetilde\gamma^-\left(t,\lambda^+_1,0\right)$]
 {
  \scalebox{0.4}[0.4]{
   \includegraphics{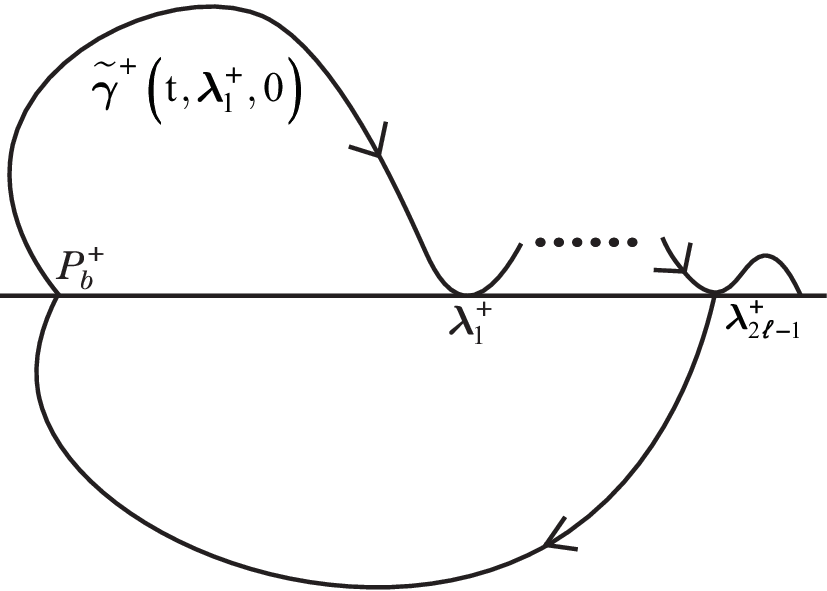}}}
 \caption{$L^{cri}(\ell)$ for $y_0<0$}
\label{Fig-thm3-1}
\end{figure}

For the other 13 types of loop $L^{cri}_*$ as shown in Figure~\ref{Fig-Loop}(b)-(n), we take ${\boldsymbol \lambda}^-={\boldsymbol 0}$ and
\begin{equation*}
\left\{
\begin{aligned}
&\lambda^+_1<...<\lambda^+_{m^+}<0,&&{\rm~for~}1\le m^+\le 4,\\
&\lambda^+_i=\left(i-m^+-1\right)\delta,&&{\rm~for~}m^+\ge 5
\end{aligned}
\right.
\end{equation*}
if $L^{cri}_*$ is the type of Figure~\ref{Fig-Loop}(c)-(e)(g)-(n) and
\begin{equation*}
\left\{
\begin{aligned}
&0<\lambda^+_1<...<\lambda^+_{m^+},&&{\rm~for~}1\le m^+\le 4,\\
&\lambda^+_i=i\delta,&&{\rm~for~}m^+\ge 5
\end{aligned}
\right.
\end{equation*}
if $L^{cri}_*$ is the type of Figure~\ref{Fig-Loop}(b)(f).
Then conclusions (a) and (b) can be proved similarly and we omit the statements.
\end{proof}

To prove Theorems~\ref{thm4} and \ref{thm5}, we need a transition map for tangent case, i.t.,
orbit $\gamma(t,x_0,y_0)$ is tangent with $S_0$ and transversally intersects $S_1$ as shown in Figure~\ref{Fig-Exam}.
In tangent case, $V_1=0$ because $\Delta_0=0$ and $\Delta_1\ne0$, which means that transition map
$V(r;x_0,y_0,T,S_0,S_1)$ starts at least degree $2$.
This is different from the classic transition map given in \eqref{V-Expan}.

\begin{lm} Assume that orbit  $\gamma\left(t,x_0,y_0\right)$ of system~\eqref{trans-sys}
is tangent with horizontal line $S_0$ at $(x_0,y_0)$ and transversally intersects $S_1$ at $(x_1,y_1)$ at time $t=T$ as shown in Figure~\ref{Fig-Exam}. If there exists an integer $m\ge 1$ such that
\begin{equation*}
g(x_0,y_0)=\frac{\partial g}{\partial x}(x_0,y_0)=...=\frac{\partial^{m-1} g}{\partial x^{m-1}}(x_0,y_0)=0,~~\frac{\partial^{m} g}{\partial x^{m}}(x_0,y_0)\ne0,
\end{equation*}
then transition map
\begin{equation*}
V\left(r;x_0,y_0,T,S_0,S_1\right)=V_{m+1}r^{m+1}+O\left(r^{m+2}\right),
\end{equation*}
where
\begin{equation*}
V_{m+1}:=\frac{\partial^m g}{\partial x^m}(x_0,y_0)\frac{N^{m+1}_{01}}{\left(m+1\right)!\Delta_1}\exp\left\{\int_0^{T}\frac{\partial f\left(\gamma\left(s,x_0,y_0\right)\right)}{\partial x}
+\frac{\partial g\left(\gamma\left(s,x_0,y_0\right)\right)}{\partial y}ds\right\}\ne0
\end{equation*}
and $N_{01}$ is the first component of unit vector $N_0$ parallel to $S_0$ and
$\Delta_1$ is defined in \eqref{D0D1}.
\label{plm1}
\end{lm}

\begin{figure}[h]
\centering
\includegraphics[width=0.5\textwidth]{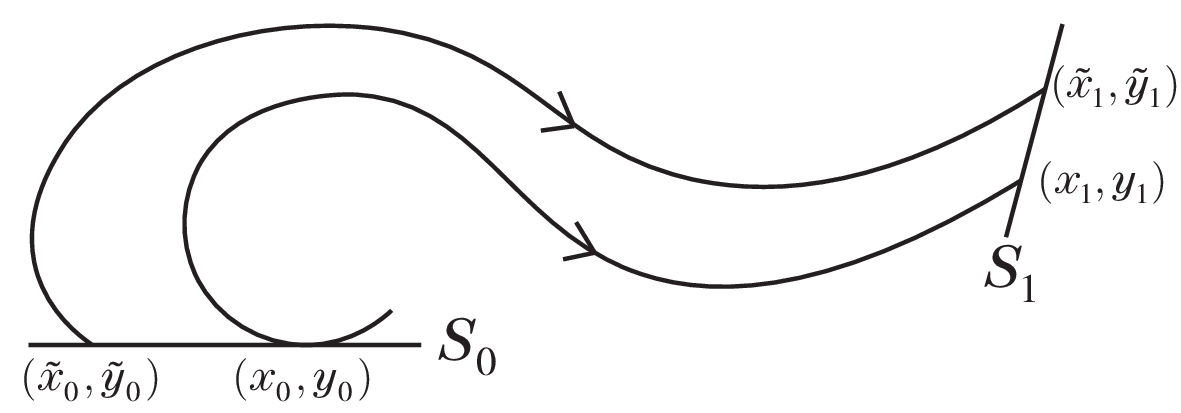}
\caption{transition map for tangent case}
\label{Fig-Exam}
\end{figure}

\begin{proof}
By \cite{Ponce15,JKHaleBook,KuznetsovBook,ZhangJ98} for any small $r\ne0$
\begin{equation}
\delta_1(r)\frac{dV(r)}{dr} = \delta_0(r)E(r),
\label{plm1-equa1}
\end{equation}
where
\begin{equation*}
E(r):=\exp\left\{\int_0^{T(r)}\frac{\partial f\left(\gamma\left(s,x_0+N_{01}r,y_0\right)\right)}{\partial x}+\frac{\partial g\left(\gamma\left(s,x_0+N_{01}r,y_0\right)\right)}{\partial y}ds\right\}
\end{equation*}
and $\delta_1(r):=\det\left({\cal Z}(x_1+V(r)N_{11},y_1+V(r)N_{12}),N_1\right), \delta_0(r):=N_{01}g(x_0+N_{01}r,y_0)$.
Then, from \eqref{plm1-equa1}
\begin{equation}
\sum_{i=0}^{k}C^i_k \frac{d^{i}\delta_1(r)}{dr^{i}}\frac{d^{1+k-i}V(r)}{dr^{1+k-i}}=\sum_{i=0}^{k}C^i_k \frac{d^{i}\delta_0(r)}{dr^{i}}\frac{d^{k-i}E(r)}{dr^{k-i}}, ~~~k\ge 1.
\label{k-der}
\end{equation}
Associated with
\begin{equation*}
\frac{d^{i}\delta_0(0)}{dr^{i}}=N^{i+1}_{01}\frac{\partial^{i} g }{\partial x^i}(x_0,y_0)=0, ~~~i=0,...,m-1,
\end{equation*}
we get
\begin{equation*}
\frac{d^{i+1} V(0)}{d r^{i+1}}=0, ~~~i=0,...,m-1.
\end{equation*}
Further, taking $k=m$ we get from \eqref{k-der} that
\begin{equation*}
\frac{d^{m+1} V(0) }{d r^{m+1}}=\frac{d^{m}\delta_0(0)}{dr^{m}}\frac{E(0)}{\delta_1(0)}=\frac{\partial^m g}{\partial x^m}(x_0,y_0)\frac{N^{m+1}_{01}}{\Delta_1}E(0)\ne 0.
\end{equation*}
This lemma is proved.
\end{proof}

Note that Lemma~\ref{plm1} can be regarded as a generalization from $m=1$ (see \cite[Proposition 1.]{Ponce15} and \cite[Lemma~2.3]{Han13}))
to general $m\ge 1$. Lemma~\ref{plm1} states the dependence of the degeneration multiplicity of transition map on the degeneration multiplicity
of tangent point $(x_0,y_0)$, which helps us prove Theorems~\ref{thm4} and \ref{thm5}.

\begin{proof}[Proof of Theorem~\ref{thm4}]
Since $m^\pm\ge 5$, we only need to consider $L^{cri}_*$ being of $10$ types as shown in Figure~\ref{Fig-Loop}(a)-(d),(f)-(h),(j)(k)(m).
Without loss of generality, we assume that $m^+\ge m^-$ and begin with $L^{cri}_*$ shown in Figure~\ref{Fig-Loop}(a).
In this case, $m^\pm$ are both odd.
We take ${\boldsymbol \lambda}^-={\boldsymbol 0}, \lambda^+_i=\left(i-m^+-1\right)\delta$ ($i=1,...,m^+$)
for small $\delta>0$, small sectors
$S_V=\left\{(0,y):~y\in(-\epsilon,\epsilon)\right\}$, $S_H=\left\{(x,0):~x\in(-\epsilon,0)\right\}$
and $S_1$ as in the first paragraph of the proof of Theorem~\ref{thm3}.

{\it Step 1. Construct function $\psi^-\left(x,{\boldsymbol k}^-\right)$.}
Orbit $\widehat\gamma^+\left(t,\lambda^+_1,\delta^5\right)$ transversally intersects $\Sigma$ at $\left(P^+,0\right)\in S_1$ for backward direction as $\delta$ is sufficiently small. On the other hand, orbit $\widehat\gamma^-\left(t,0,0\right)$ transversally intersects $\Sigma$ at $\left(P_x,0\right)$ for forward direction.
As in the proof of Theorem~\ref{thm3}, there is one $y_0\in(-\epsilon,\epsilon)$ such that
$\widehat\gamma^-\left(t,0,y_0\right)$ transversally intersects $\Sigma$ at $\left(P^+,0\right)$ for forward direction.
Define function $\psi^-\left(x,{\boldsymbol k}^-\right)$ via~\eqref{dkr12}. By Proposition~\ref{prop2} orbit $\widetilde\gamma^-\left(t,0,0\right)$ transversally intersects
$\Sigma$ at $\left(P^+,0\right)$ for forward direction.

{\it Step 2. Define the displacement function.}
For any $x\in[\lambda^+_1,0]$, orbit $\widetilde\gamma^-\left(t,x,0\right)$ transversally intersects $\Sigma$ at $\left(P^-(x),0\right)$ for forward direction and
\begin{equation*}
P^-(x)=P^++V^-(x;0,0,T^-_1,S_H,S_1)=P^++V^-_{m^-+1}x^{m^-+1}+O\left(x^{m^-+2}\right)
\end{equation*}
by Lemma~\ref{plm1}, where $T^-$ satisfies $\widetilde\gamma^-\left(T^-,0,0\right)=(P^+,0)^\top$. On the other hand, for any $\alpha\in[\lambda^+_1,0]$ orbit $\widehat\gamma^+\left(t,P^-(x),0\right)$ transversally intersects sector $S(\alpha)$ at $\left(\alpha,Y(P^-(x),\alpha)\right)$ for forward direction, where $S(\alpha):=\left\{\left(\alpha,y\right):~y\in(-1,1)\right\}$.
Further, for $(x,\alpha)\in[\lambda^+_1,0]\times[\lambda^+_1,0]$ we have
\begin{eqnarray*}
\begin{aligned}
Y\left(P^-(x),\alpha\right)
& = Y(P^+,\alpha)+V^+(P^-(x);P^+,0,T^+(\alpha),S_1,S(\alpha))\\
& = Y(P^+,\alpha)+V^+_1(\alpha)\left(P^-(x)-P^+\right)+O\left(\left(P^-(x)-P^+\right)^2\right)\\
& = Y(P^+,\alpha)+V^+_1(\alpha)V^-_{m^-+1}x^{m^-+1}+O\left(x^{m^-+2}\right)
\end{aligned}
\end{eqnarray*}
by Lemma~\ref{plm1}, where $T^+(\alpha)$ satisfies that $\widetilde\gamma^+\left(T^+(\alpha),P^+,0\right)=\left(\alpha,Y(P^+,\alpha)\right)^\top$. Note that $Y\left(P^+,\lambda^+_1\right)=\delta^5$ and for any $\alpha\in(\lambda^+_1,0)$ we obtain $Y\left(P^+,\alpha\right)=\delta^5+o\left(\delta^5\right)$ by \eqref{Theta-Esti}. Then by $V^+_1\left(\lambda^+_1\right)\ne0$,
\begin{equation}
Y\left(P^-(x),\alpha\right) = \delta^5+o\left(\delta^5\right)>0
\label{Estimate-Y}
\end{equation}
for sufficiently small $\delta$. The displacement function ${\cal Y}(x)$ is defined over interval $[\lambda^+_1,0]$ as
\begin{equation}
{\cal Y}(x):=Y\left(P^-(x),x\right)-\psi^+\left(x,{\boldsymbol k}^+\right),
\label{Dis-Y}
\end{equation}
where $\psi^+\left(x,{\boldsymbol k}^+\right)$ is undetermined and satisfies the requirement in Proposition~\ref{prop1}.
Thus, a zero $x_0$ of ${\cal Y}(x)$ corresponds to one loop of system~\eqref{pws4} if $Y\left(P^-(x_0),x\right)-\psi^+\left(x,{\boldsymbol k}^+\right)>0$ for $x\in[\lambda^+_1,x_0)$ by Proposition~\ref{prop2}.

{\it Step 3. Construct function $\psi^+\left(x,{\boldsymbol k}^+\right)$.}
Function $\psi^+\left(x,{\boldsymbol k}^+\right)$ is firstly defined by taking $d=(m^++1)/2,~k^+_1=-(m^++2)\delta,~k^+_{2d+1}=0$ and
\begin{equation*}
\left\{
\begin{aligned}
&k^+_i=\lambda^+_{i-1},&&{\rm~for~}i=2,...,2d,\\
&k^+_{i+2d+1}=0,&&{\rm~for~}i=1,...,d,
\end{aligned}
\right.
\end{equation*}
i.e., $\psi^+\left(x,{\boldsymbol k}^+\right)\equiv 0$.
In the following we change the values of $k^+_{i+2d+1}$ ($i=1,...,d$) to obtain the coexistence of bifurcating
$L^{cri}(1)$ and bifurcating $L^{cro}(1)$.
We only need to consider three cases:
$$
{\rm (C1)} ~\ell=(m^+-1)/2,~~~~{\rm (C2)} ~\ell=(m^+-1)/2-1,~~~~{\rm (C3)} ~0\le \ell\le (m^+-1)/2-2.
$$

For case (C1), we take $k^+_{i+2d+1}=Y\left(P^-\left(\lambda^+_{2i-1}\right),\lambda^+_{2i-1}\right)$ for $i=1,...,d$.
Clearly, ${\boldsymbol k}^+$ satisfies \eqref{K} by \eqref{Estimate-Y} and $\lambda^+_{2i-1}$ ($i=1,...,(m^++1)/2$) are zeros of ${\cal Y}(x)$ defined in \eqref{Dis-Y}.
These $\ell+1$ zeros correspond to $\ell+1$ bifurcating $L^{cri}(1)$.
Therefore, system~\eqref{pws4} has $\ell+1$ bifurcating $L^{cri}(1)$ connecting $\left(\lambda^+_{2i-1},0\right)$ ($i=1,...,(m^++1)/2$) respectively.
Clearly, there is no bifurcating $L^{cro}(1)$ because there is no more visible tangent points outside these
bifurcating $L^{cri}(1)$. Case (C1) is proved and an example for $m^+=5$ is shown in Figure~\ref{Fig-thm4-Exam1}(a).

\begin{figure}[htp]
\centering
\subfigure[three $L^{cri}(1)$]
 {
  \scalebox{0.4}[0.4]{
   \includegraphics{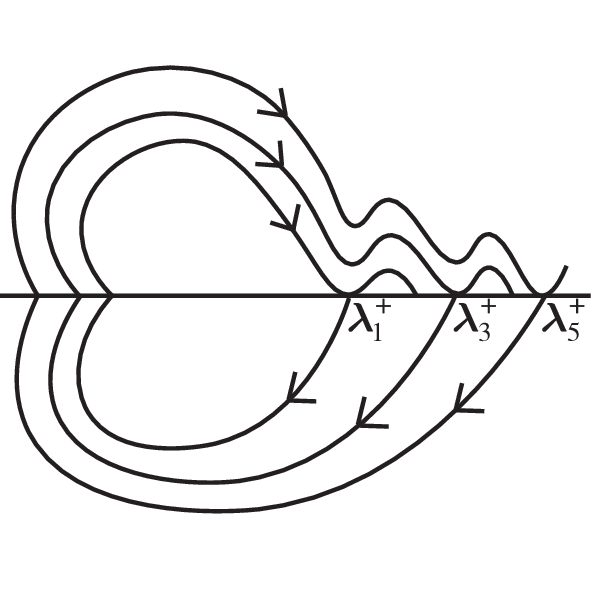}}}~~~
\subfigure[two $L^{cri}(1)$, one $L^{cro}(1)$]
 {
  \scalebox{0.4}[0.4]{
   \includegraphics{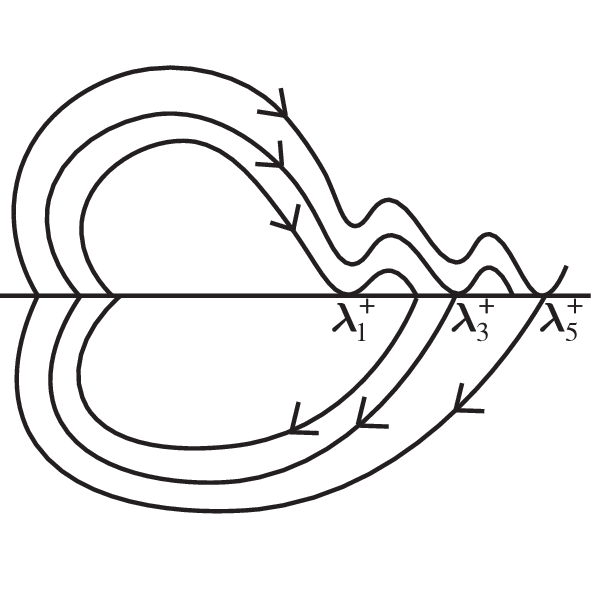}}}~~~
   \subfigure[one $L^{cri}(1)$, two $L^{cro}(1)$]
 {
  \scalebox{0.4}[0.4]{
   \includegraphics{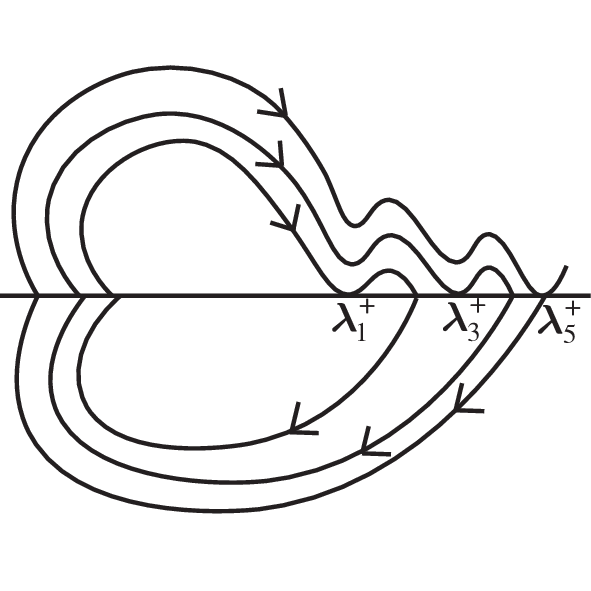}}}
 \caption{bifurcating $L^{cri}(1)$ and bifurcating $L^{cro}(1)$}
\label{Fig-thm4-Exam1}
\end{figure}

For case (C2), we take $k^+_{1+2d+1}=0$ and
$k^+_{i+2d+1}=Y\left(P^-\left(\lambda^+_{2i-1}\right),\lambda^+_{2i-1}\right)$ for $i=2,...,d$,
which implies that system~\eqref{pws4} has $(m^+-1)/2$ bifurcating $L^{cri}(1)$ connecting $\left(\lambda^+_{2i-1},0\right)$ ($i=2,...,(m^++1)/2$) respectively. Consider the interval $\left[\lambda^+_2,\lambda^+_3\right]$, by \eqref{Estimate-Y} we obtain that ${\cal Y}\left(\lambda^+_{2}\right)=Y\left(P^-\left(\lambda^+_2\right),\lambda^+_2\right)>0$. Thus, orbit $\widetilde\gamma^+\left(P^-\left(\lambda^+_2\right),0\right)$ transversally intersects $\Sigma$ at $\left(\widetilde\lambda^+_2,0\right)$ for forward direction and $\lambda^+_2<\widetilde\lambda^+_2<\lambda^+_3$.
On the other hand, there is a bifurcating $L^{cri}(1)$ connecting $(\lambda^+_3,0)$. Further, for any $P_3$ satisfying $0<\lambda^+_{3}-P_3\ll 1$ orbit $\widetilde\gamma^+\left(t,P^-\left(P_3\right),0\right)$ transversally intersects $\Sigma$ at $\left(\widetilde P_3,0\right)$. This bifurcating $L^{cri}(1)$ is unstable by \cite{Ponce15,Han13},
which implies that $P_3>\widetilde P_3$. Thus, system~\eqref{pws4} has a crossing limit cycle intersecting $\Sigma$ transversally
at $\left(Q_2,0\right)$ and $Q_2\in(\lambda^+_2,P_3)$.
Then we take $k^+_{1+2d+1}=Y\left(P^-(Q_2),\lambda^+_1\right)$ and consequently, system~\eqref{pws4} still has $(m^+-1)/2$ bifurcating $L^{cri}(1)$ connecting $\left(\lambda^+_{2i-1},0\right)$ ($i=2,...,(m^++1)/2$) respectively and $1$ bifurcating $L^{cro}(1)$ connecting $\left(\lambda^+_1,0\right)$. The case (C2) is proved and an example for $m^+=5$ is shown in Figure~\ref{Fig-thm4-Exam1}(b).

For case (C3), we take $k^+_{i+2d+1}=Y\left(P^-\left(\lambda^+_{2i-1}\right),\lambda^+_{2i-1}\right)$ for $i=n,...,d$,
where $n:=(m^++1)/2-\ell\ge 3$.
As the above, system~\eqref{pws4} has $d-n-1$ (i.e., $\ell+1$) bifurcating $L^{cri}(1)$.
It can be proved by a similar way that system~\eqref{pws4} has a crossing limit cycle transversally intersecting $\Sigma$ at $(Q_{2n-2},0)$ satisfying that $Q_{2n-2}\in\left(\lambda^+_{2n-2},\lambda^+_{2n-1}\right)$. By taking $k^+_{n-1+2d+1}=Y\left(P^-\left(Q_{2n-2}\right),\lambda^+_{2n-3}\right)$, system~\eqref{pws4} has $1$ bifurcating $L^{cro}(1)$ connecting $\left(\lambda^+_{2n-3},0\right)$. Since ${\cal Y}\left(\lambda^+_{2n-4}\right) > 0$ and ${\cal Y}\left(\lambda^+_{2n-3}\right) < 0$,
system~\eqref{pws4} has a crossing limit cycle transversally intersecting $\Sigma$ at $\left(Q_{2n-4},0\right)$ and $Q_{2n-4}\in\left(\lambda^+_{2n-4},\lambda^+_{2n-3}\right)$. By taking $k^+_{n-2+2d+1}=Y\left(P^-\left(Q_{2n-4}\right),\lambda^+_{2n-5}\right)$, system~\eqref{pws4} has $1$ bifurcating $L^{cro}(1)$ connecting $\left(\lambda^+_{2n-5},0\right)$, which finishes the proof of $n=3$. If $n>3$, corresponding $k^+_{i+2d+1}$ ($i=1,...,n-3$) is obtained one by one via analysis of $k^+_{n-2+2d+1}$ and consequently,
system~\eqref{pws4} has $n-1$ (i.e., $(m^+-1)/2-\ell$) bifurcating $L^{cro}(1)$ connecting $\left(\lambda^+_{2i-1},0\right)$ respectively.
The case (C3) is proved and an example for $m^+=5$ is shown in Figure~\ref{Fig-thm4-Exam1}(c).

For other $9$ types of $L^{cri}_*$ shown in Figure~\ref{Fig-Loop}(b)-(d),(f)-(h),(j)(k)(m),
this theorem can be proved similarly and we omit their proofs.
\end{proof}

Not only the numbers of $L^{cri}(1)$ and $l^{cro}(1)$ are proved, but also the location of them is given in the proof of Theorem~\ref{thm4}. To prove Theorem~\ref{thm5}, we give the location of all $L^{cri}(1)$ of case (C1) as
\begin{equation}
L^{cri}_1~\hookrightarrow~L^{cri}_3~\hookrightarrow~...~\hookrightarrow~L^{cri}_{m^+-2}~\hookrightarrow~L^{cri}_{m^+},
\label{Nest1}
\end{equation}
where $L^{cri}_{2i-1}$ ($i=1,...,(m^++1)/2$) denotes the bifurcating $L^{cri}(1)$ connecting $(\lambda^+_{2i-1},0)$ and $L^{cri}_1~\hookrightarrow~L^{cri}_3$ means that $L^{cri}_1$ lies in the region surrounded by $L^{cri}_3$.

\begin{proof}[Proof of Theorem~\ref{thm5}]Without loss of generality, we assume that $m^+\ge m^-$.

{\it Step 1. Prove conclusion that $\beta_c+\beta_s\ge 1$.}
By Theorem~\ref{thm3}, for $L^{cri}_*$ shown in Figure~\ref{Fig-Loop}(a) there exists ${\boldsymbol \lambda}^\pm$ and $\psi^\pm\left(x,{\boldsymbol k}^\pm\right)$ such that system~\eqref{pws4} has a bifurcating $L^{cri}(1)$ connecting $\left(\lambda^+_1,0\right)$. Then for $k^-_4\in(y_0-\epsilon,y_0+\epsilon)$ we define displacement function
\begin{equation*}
D(x):=P^-_b(k^-_4)+V^-(x;\lambda^+_1,0,T^-_b(k^-_4),S_0,S_1)-P^+_b-V^+(x;\lambda^+_1,0,T^+_b,S_0,S_1),
\end{equation*}
where $y_0$ is the zero of $D\left(y;\lambda^+_1\right)$ defined by \eqref{Dydef2} in the proof of Theorem~\ref{thm3}, $S_0:=\left\{(x,0):~x\in(\lambda^+_1-\epsilon,\lambda^+_1]\right\}, S_1:=\left\{(x,0):~x\in(P^+_b-\epsilon,P^+_b+\epsilon)\right\}$ and $\widetilde\gamma^+\left(T^+_b,\lambda^+_1,0\right)=\left(P^+_b,0\right)^\top$, $\widetilde\gamma^+\left(T^-_b(k^-_4),\lambda^+_1,0\right)=\left(P^-_b(k^-_4),0\right)^\top$.
Taking unit vectors $N_{0,1}=(1,0)^\top$, we get
\begin{equation*}
D(x)=P^-_b(k^-_4)-P^+_b+V^-_1(k^-_4)(x-\lambda^+_1)+O\left((x-\lambda^+_1)^2\right)
\end{equation*}
by \eqref{V-Expan} and Lemma~\ref{plm1}, where $V^-_1(k^-_4)<0$. Clearly, $D(\lambda^+_1)=0$ for $k^-_4=y_0$ because $P^-_b(k^-_4)=P^+_b$. Then we take $k^-_4=y_0-\alpha$ for $0<\alpha\ll 1$ and consequently, $P^-_b\left(y_0-\alpha\right)<P^+_b$ by Proposition~\ref{prop2}. Implicit Function Theorem shows that there is one $Q_c$ such that $D(Q_c)=0$ and $Q_c<\lambda^+_1$, which implies that system~\eqref{pws4} has a crossing limit cycle intersecting $\Sigma$ at $(Q_c,0)$ as shown in Figure~\ref{Fig-thm5-1}(a). On the other hand, we take $k^-_4=y_0+\alpha$ and then orbit $\widetilde\gamma^-\left(t,P^+_b,0\right)$ intersects $\Sigma$ at $(Q_s,0)$ for backward direction and $Q_s>\lambda^+_1$. Since $\left\{(x,0):~x\in(\lambda^+_1,\lambda^+_2)\right\}\subset\Sigma_s$ and
\begin{equation*}
\lim_{x\to\lambda^+_1}\frac{\widetilde Z^+_1(x,0)\widetilde Z^-_2(x,0)-\widetilde Z^-_1(x,0)\widetilde Z^+_2(x,0)}{\widetilde Z^-_2(x,0)-\widetilde Z^+_2(x,0)} = \widetilde Z^+_1(\lambda^+_1,0)>0
\end{equation*}
by sliding vector field in \eqref{SVF}, the sliding orbit starting from $(Q_s,0)$ goes to $(\lambda^+_1,0)$ for backward direction. Thus, system~\eqref{pws4} has a bifurcating $L_s$ connecting visible bifurcating tangent point $(\lambda^+_1,0)$
as shown in Figure~\ref{Fig-thm5-1}(b), where $L_s$ is defined in Definition~\ref{df-sli}.
The analysis of other loops shown in Figure~\ref{Fig-Loop} is similar and we omit the statements.
Conclusion that $\beta_c+\beta_s\ge 1$ is proved.

\begin{figure}[htp]
\centering
\subfigure[bifurcating $L_c$]
 {
  \scalebox{0.4}[0.4]{
   \includegraphics{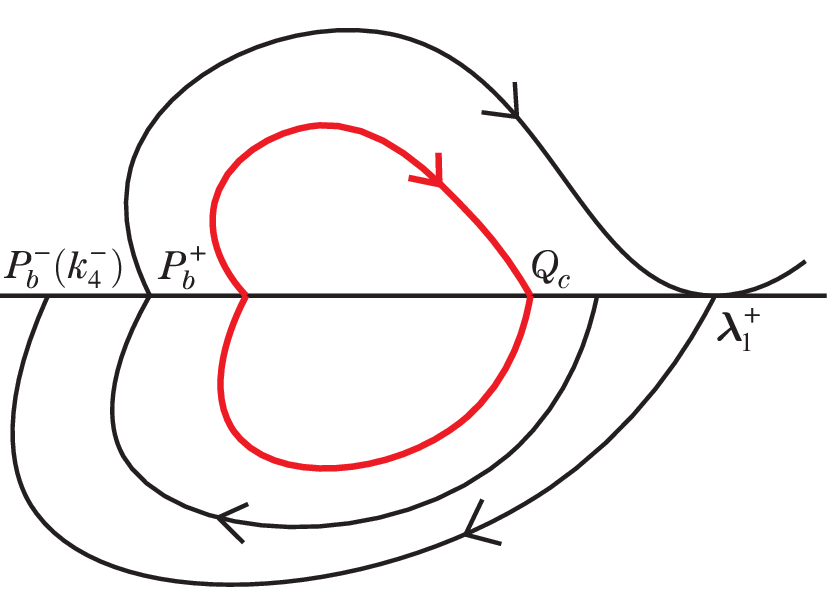}}}~~~
\subfigure[bifurcating $L_s$]
 {
  \scalebox{0.4}[0.4]{
   \includegraphics{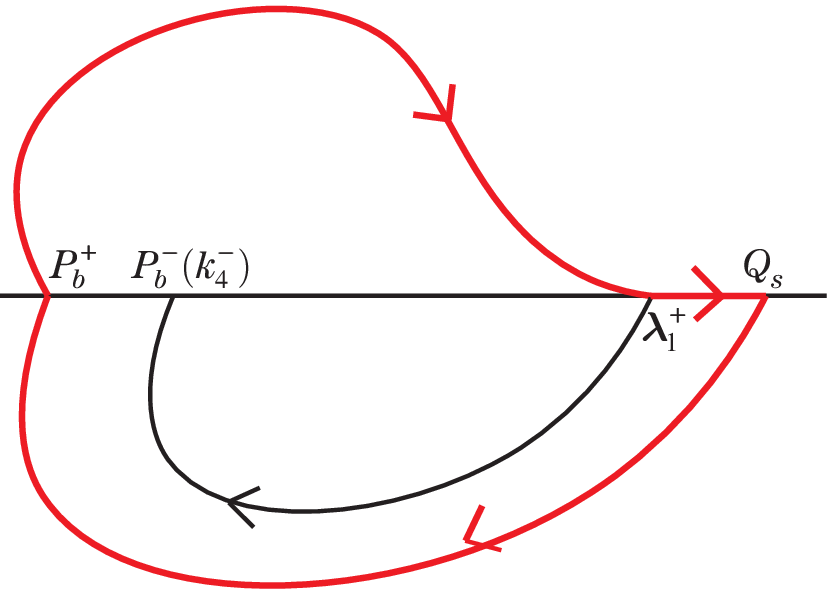}}}
 \caption{bifurcating $L_c$ and bifurcating $L_s$ for $\beta_c+\beta_s\ge 1$}
\label{Fig-thm5-1}
\end{figure}

{\it Step 2. Prove conclusion for $m^\pm\ge 5$.} We only need to consider $L^{cri}_*$ being of $10$ types as shown in Figure~\ref{Fig-Loop}(a)-(d),(f)-(h),(j)(k)(m).
For $L^{cri}_*$ shown in Figure~\ref{Fig-Loop}(a), Theorem~\ref{thm4} shows that there exists corresponding ${\boldsymbol \lambda}^\pm$ and $\psi^\pm(x,{\boldsymbol k}^\pm)$ such that location \eqref{Nest1} holds, i.e., system~\eqref{pws4} has $(m^++1)/2$ bifurcating $L^{cri}(1)$ connecting $(\lambda^+_{2i-1},0)$ respectively ($i=1,...,(m^++1)/2$). We only need to consider three cases
$${\rm~(C1)~}\ell=0,~~~~{\rm~(C2)}~\ell=(m^++1)/2,~~~~{\rm~(C3)~}1\le \ell\le (m^+-1)/2,$$

For case (C1), we take
\begin{equation*}
k^+_{i+2d+1}=Y\left(P^-\left(\lambda^+_{2i-1}\right),\lambda^+_{2i-1}\right)+\alpha_i, ~{\rm~for~}i=1,...,(m^++1)/2,
\end{equation*}
where $Y\left(P^-\left(\lambda^+_{2i-1}\right),\lambda^+_{2i-1}\right)$ and $P^-\left(\lambda^+_{2i-1}\right)$ is defined
in the first paragraph of step 2 in the proof of Theorem~\ref{thm4} and $0<\alpha_i\ll 1$.
By the analysis in step 1 system~\eqref{pws4} has $(m^++1)/2$ bifurcating $L_s$ connecting $\left(\lambda^+_{2i-1},0\right)$ respectively.
For any $x\in\left[\lambda^+_{2i},\lambda^+_{2i+1}\right]$ ($i=1,...,(m^+-1)/2$), orbit $\widetilde\gamma^-\left(t,x,0\right)$ transversally intersects $\Sigma$ at $\left(P^-(x),0\right)$ for forward direction and then orbit $\widetilde\gamma^+\left(t,P^-(x),0\right)$ transversally intersects $\Sigma$ at $\left(\widetilde x,0\right)$ for forward direction. It is not hard to check that $\lambda^+_{2i}<\widetilde \lambda^+_{2i}$ and $\lambda^+_{2i+1}>\widetilde \lambda^+_{2i+1}$, which implies that there is a $Q_{2i}\in\left[\lambda^+_{2i},\lambda^+_{2i+1}\right]$ such that $Q_{2i}=\widetilde Q_{2i}$. Thus, system~\eqref{pws4} has a crossing limit cycle transversally intersecting $\Sigma$ at $(Q_{2i},0)$ and lying in the region surrounded by $L^s_{2i-1}$ and $L^s_{2i+1}$, where $L^s_{2i-1}$ denotes the bifurcating $L_s$ connecting $(\lambda^+_{2i-1},0)$ ($i=1,...,(m^++1)/2)$.
Therefore, $\beta_c\ge (m^++1)/2-1$ and $\beta_s=(m^++1)/2$. The proof is finished
and an example for $m^+=5$ is shown in Figure~\ref{Fig-thm5-2}(a).

\begin{figure}[htp]
\centering
\subfigure[two $L_c$, three $L_s$]
 {
  \scalebox{0.4}[0.4]{
   \includegraphics{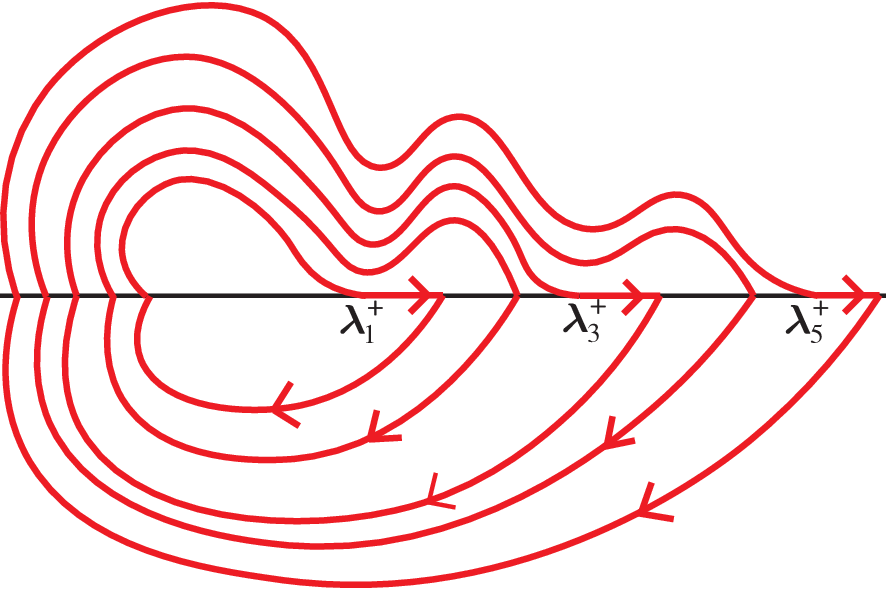}}}
\subfigure[five $L_c$]
 {
  \scalebox{0.4}[0.4]{
   \includegraphics{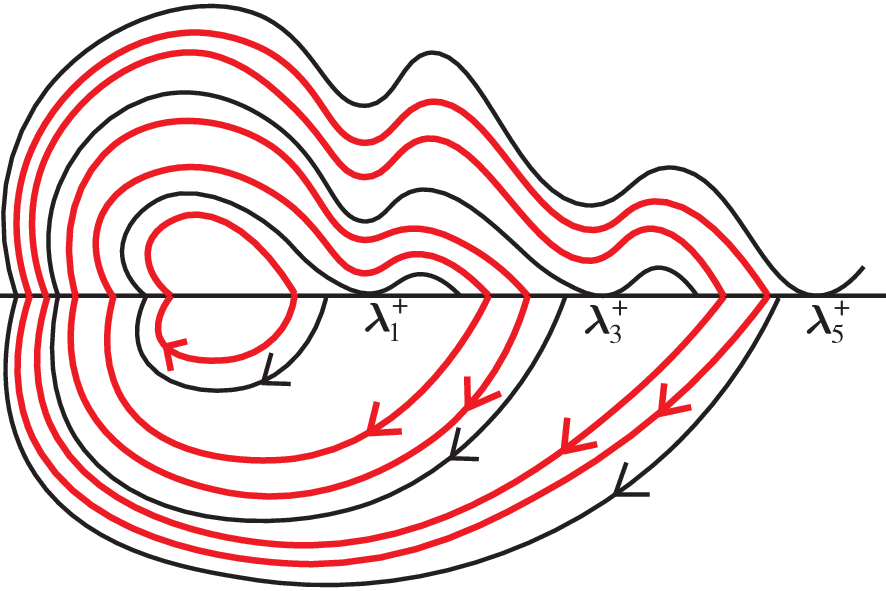}}}
\subfigure[four $L_c$, one $L_s$]
{
  \scalebox{0.4}[0.4]{
   \includegraphics{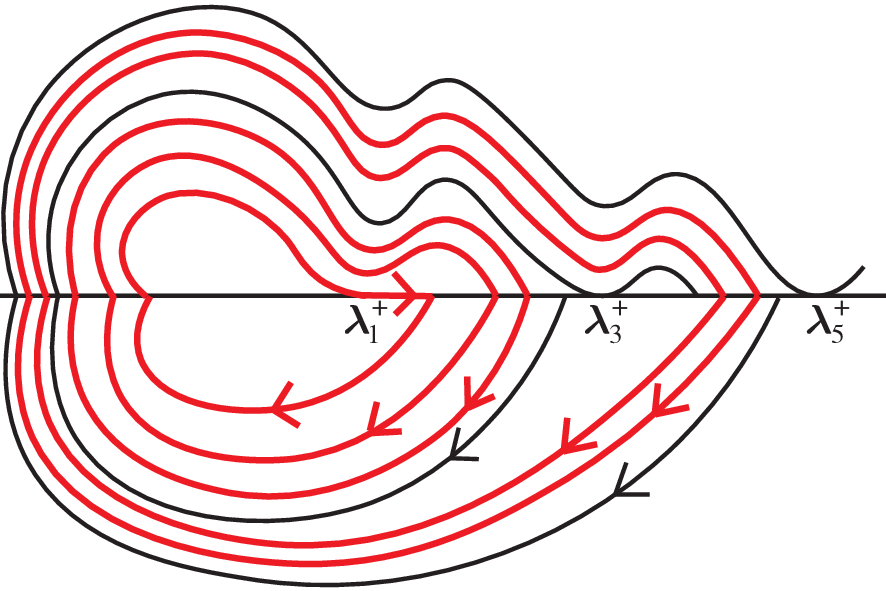}}}
\subfigure[three $L_c$, two $L_s$]
 {
  \scalebox{0.4}[0.4]{
   \includegraphics{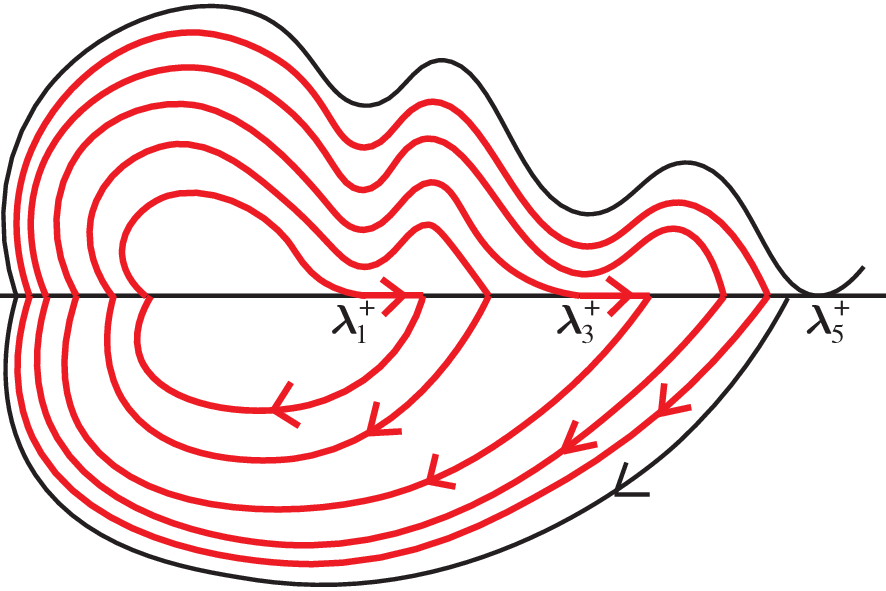}}}
 \caption{bifurcating $L_c$ and bifurcating $L_s$ for $\beta_c+\beta_s\ge m^*$}
\label{Fig-thm5-2}
\end{figure}

For case (C2), we take
\begin{equation*}
k^+_{i+2d+1}=Y\left(P^-\left(\lambda^+_{2i-1}\right),\lambda^+_{2i-1}\right)-\alpha_i,~{\rm~for~}i=1,...,(m^++1)/2,
\end{equation*}
which implies that system~\eqref{pws4} has $(m^++1)/2$ crossing limit cycles
(denoted by $L^c_{2i-1}$) bifurcating from $L^{cri}_{2i-1}$ respectively
by step 1 and no bifurcating $L_s$ appears.
Since all of these crossing limit cycle are unstable, system~\eqref{pws4} has a crossing limit cycle lying in the region surrounded by $L^c_{2i-1}$ and $L^c_{2i+1}$ respectively. Then we get $(m^+-1)/2$ bifurcating $L_c$ again.
Therefore, $\beta_c\ge (m^++1)/2+(m^+-1)/2=m^+$ and $\beta_s=0$.
The proof is finished and an example for $m^+=5$ is shown in Figure~\ref{Fig-thm5-2}(b).

For case (C3), we take
\begin{equation*}
\left\{
\begin{aligned}
&k^+_{i+2d+1}=Y\left(P^-\left(\lambda^+_{2i-1}\right),\lambda^+_{2i-1}\right)+\alpha_i,&&{\rm~for~}i=1,...,n,\\
&k^+_{i+2d+1}=Y\left(P^-\left(\lambda^+_{2i-1}\right),\lambda^+_{2i-1}\right)-\alpha_i,&&{\rm~for~}i=n+1,...,(m^++1)/2,\\
\end{aligned}
\right.
\end{equation*}
where $n:=(m^++1)/2-\ell$. Thus, system~\eqref{pws4} has $n$ (i.e., $(m^++1)/2-\ell$)
bifurcating $L_s$ connecting $\left(\lambda^+_{2i-1},0\right)$ ($i=1,...,n$) respectively as in case (C1) of step 2.
Moreover, system~\eqref{pws4} has $(m^++1)/2-n$ (i.e., $\ell$) unstable
crossing limit cycle bifurcating from
$L^{cri}_{2i-1}$ ($i=n+1,...,(m^++1)/2$) respectively as in case (C2) of step 2.
Analyzing the region surrounded by $L^s_{2n-1}$ and $L^c_{2n+1}$,
we obtain that $\lambda^+_{2n}<\widetilde\lambda^+_{2n}$. Associated with the instability of $L^c_{2n+1}$,
system~\eqref{pws4} has a crossing limit cycle lying in this region.
On the other hand, by the analysis in (C1) (resp. (C2)) there are at least $n-1$ (resp. $(m^++1)/2-n-1$)
bifurcating $L_c$ in the region surrounded by $L^s_1$ and $L^s_{2n-1}$ (resp. $L^c_{2n+1}$ and $L^c_{m^+}$).
That is, we get $1+(n-1)+(m^++1)/2-n-1=(m^+-1)/2$ bifurcating $L_c$ again.
Therefore, $\beta_c\ge (m^+-1)/2+(m^++1)/2-n=m^+-n$ and $\beta_s=n$.
The proof is finished and an example for $m^+=5$ is shown in Figure~\ref{Fig-thm5-2}(c)(d).

For other $9$ types of $L^{cri}_*$ shown in Figure~\ref{Fig-Loop}(b)-(d),(f)-(h),(j)(k)(m) in step 2,
this theorem can be proved similarly and we omit their proofs.
\end{proof}

%%%%%%%%%%%%%%%%%%%%%%%%%%%%%%%%%%%%%%%%%%%%%%%%%%%%%%%%%%%
\section{Conclusions and remarks}
\setcounter{equation}{0}
\setcounter{lm}{0}
\setcounter{thm}{0}
\setcounter{rmk}{0}
\setcounter{df}{0}
\setcounter{cor}{0}

In this paper, we come up with a series definitions and complete classifications for tangent points and loops connecting them. With these classifications, we mainly investigate question (Q).
For the part of tangent points in question (Q), we focus on the numbers of bifurcating tangent points and bifurcating tangent orbits
from tangent point $O$ of any multiplicity.
Corresponding result of the number of bifurcating tangent points is stated in Theorem~\ref{thm1},
which generalizes the results for tangent point $O$ of multiplicities $1$ and $2$ given in
previous publications such as \cite{Fang21,Teixeira11,Kuznetsov03,Li20,Han13}.
Corresponding result of the number of bifurcating tangent orbits is stated in Theorem~\ref{thm2},
which is new and important for the bifurcations of loops connecting tangent points.
For the part of loops connecting tangent points in question (Q),
we focus on the numbers of bifurcating loops and bifurcating crossing limit cycles
from critical loops $L^{cri}_*$
connecting tangent point $O$ of any multiplicity as shown in Figure~\ref{Fig-Loop}.
Corresponding results are stated in Theorems~\ref{thm3}, \ref{thm4} and \ref{thm5}.
Since the multiplicity of tangent point $O$ is general and not required to be $1$ or $2$ as
in previous publications such as \cite{Ponce15,Kuznetsov03,Han13,Novaes18,Huang22},
classic analysis is no longer valid and we have to develop new methods to deal with these two questions.

For bifurcating tangent points from tangent point $O$, it is directly analyzed by applying Malgrange Preparation
Theorem because they are zeros of a function $h(x)$ defined in section 1 by
the switching vector field on switching manifold $\Sigma$.
Thus, for system~\eqref{pws3} with tangent point $O$ of multiplicity $(m^+,m^-)$ we construct an
unfolding~\eqref{pws4} which concludes functional parameters ${\boldsymbol \lambda}^\pm$ and functional functions $\psi^\pm(x,{\boldsymbol k}^\pm)$. As in the proof of Theorem~\ref{thm2}, ${\boldsymbol \lambda}^\pm$ are used to determine the number, location and
multiplicity of bifurcating tangent points, which is also shown by $\alpha^\pm$ in unfolding~\eqref{pws2-exunfold} in the proof of Theorem~\ref{thm1}. Functional functions $\psi^\pm(x,{\boldsymbol k}^\pm)$ are explained by
Proposition~\ref{prop2} and are used to determine the tangent intersections between orbits and switching manifold
$\Sigma$ in the proof of Theorem~\ref{thm2}.

For bifurcating loops from $L^{cri}_*$, classic analysis method is defining displacement function
 with a sector on switching manifold $\Sigma$ as its domain, and then finding zeros by Intermediate Value
 Theorem or Implicit Function Theorem and so on. This is valid for the case that multiplicity of tangent
 point $O$ is $1$ because its structurally stability leads to that the
 dynamical behaviors of bifurcating tangent points and bifurcating tangent orbits
are clear under perturbations.
 However, when the multiplicity of tangent point $O$ is higher, the classic definition of displacement function
 is no longer valid because in small neighborhood of $O$ several new tangent points bifurcate under perturbations.
 In order to deal with it, we use ${\boldsymbol \lambda}^\pm$ to desingularize tangent point $O$ and then
 use functional functions $\psi^\pm(x,{\boldsymbol k}^\pm)$ to define new
 displacement function ${\cal Y}(x)$ as \eqref{Dis-Y}, which is determined by the orbit starting at $(x,0)$ and ending at $(x,{\cal Y}(x))$.
 The relation between degeneration of tangent point $O$ and the numbers of bifurcating loops from $L^{cri}_*$ are
 given in Theorems~\ref{thm3}, \ref{thm4} and \ref{thm5}, which shows that loops $L^{cri}_*$ own partly properties
 of limit cycles and homoclinic loops.

As discussions in section 1, a usual way in the investigation of bifurcations for piecewise-smooth systems is constructing functional parameters
(see, e.g., \cite{Bonet18,Fang21,Teixeira11,Kuznetsov03,Novaes18,Huang22}).
Unfortunately, there is no general methods to construct enough functional parameters independent among themselves in
perturbation systems.
Up to now there are no more than $2$ functional parameters in previous publications.
This leads to that for systems in high degeneration case (e.g., having tangent points of high multiplicities)
two functional parameters are not enough to unfold plentiful dynamics sufficiently.
Another thing, functional parameters influence the dynamics but usually have
no explicit expressions in perturbation vector fields, which means that we do not know how to
perturb vector fields to exhibit dynamical behaviors.
Thus, we not only construct functional parameters ${\boldsymbol \lambda}^\pm$ but also
construct functional functions $\psi^\pm(x,{\boldsymbol k}^\pm)$.
Moreover, functional parameters and functional functions are all expressed explicitly in the
perturbation vector field in \eqref{pws4}, where $\psi^\pm(x,{\boldsymbol k}^\pm)$ have expressions
as defined in Proposition~\ref{prop1}.

We have to remark that system~\eqref{pws3} is a special system~\eqref{pws2} because
items $\Upsilon^\pm(x,y)y$ are ignored.
In the case of no $\Upsilon^\pm(x,y)y$, we successfully construct functional functions $\psi^\pm(x,{\boldsymbol k}^\pm)$
in unfolding~\eqref{pws4}.
Unfortunately, the existence of $\Upsilon^\pm(x,y)y$ makes that
the horizontal isoclines are some curves and no longer straight lines $x=\lambda^+_i$ ($i=1,...,m^+$).
Our construction method for $\psi^\pm(x,{\boldsymbol k}^\pm)$ is invalid for such case and, hence,
a new unfolding method for system~\eqref{pws2} is needed.

%%%%%%%%%%%%%%%%%%%%%%%%%%%%%%%%%%%%%%%%%%%%%%%%%%%%%%%%%%%%%%%%%

\end{document}